\documentclass[final,14p,times]{elsarticle}
\usepackage{geometry}
\usepackage{graphicx,subfigure,epsfig}

\usepackage{amsmath,mathrsfs,amssymb,amsthm,dsfont,bm,amsfonts}      %
\usepackage{epstopdf}
\usepackage{color}
\usepackage{float}
\allowdisplaybreaks

\newtheorem{thm}{Theorem}[section]
\newtheorem{lem}{Lemma}[section]

\newcommand{\p}{\partial}
\newcommand{\f}{\frac}

\newcommand{\up}{\textup}
\newcommand{\tr}{\triangle}

\journal{*******************}
\begin{document}

\begin{frontmatter}

\title{Numerical analysis of growth-mediated autochemotactic pattern formation in self-propelling bacteria}

\author[Aaddress]{Jiansong Zhang}
\ead{jszhang@upc.edu.cn}
\author[Baddress]{Maosheng Jiang}
\ead{msjiang@qdu.edu.cn}
\author[Caddress]{Jiang Zhu}
\ead{jiang@lncc.br}
\author[Daddress]{Xijun Yu}
\ead{yuxj@iapcm.ac.cn}
\author[Eaddress]{Luiz Bevilacqua}
\ead{bevilacqua@coc.ufrj.br}

\address[Aaddress]{Department of Applied Mathematics,
        China University of Petroleum, Qingdao 266580, China}
\address[Baddress]{School of Mathematics and Statistics, Qingdao University, Qingdao 266071,  China.}
\address[Caddress]{Laborat\'{o}rio Nacional de Computa\c{c}\~{a}o Cient\'{i}fica, MCTI, Avenida Get\'{u}lio Vargas 333,  Petr\'{o}polis, 25651-075, RJ, Brazil}

\address[Daddress]{Laboratory of Computational Physics, Institute of Applied Physics and Computational Mathematics, Beijing 100088,  China}
 \address[Eaddress]{A.L. Coimbra Institute, COPPE, Federal University of Rio de Janeiro, Rio de Janeiro, Brazil}

\cortext[io]{Corresponding author. Maosheng Jiang}

\begin{abstract}
In this paper,  a decoupled characteristic Galerkin finite element procedure is provided for simulating growth-mediated autochemotactic pattern formation in self-propelling bacteria.  In this procedure,  a modified characteristic Galerkin method is established to solve the bacterial density equation, while  the classical finite element procedure is considered for the self-secreted chemical density and polarization dynamics equations system.   The convergence of this proposed method is considered  under some regularity assumptions and the corresponding error estimate is derived. Numerical experiments are carried out to support the theoretical analysis. Furthermore, several new wave type pattern formations are found.
\end{abstract}


\end{frontmatter}

% =========================================================

\section{Introduction}

In order to investigate the spatiotemporal dynamics of a microbial colony\cite{MP2018,Toner1995,Budrene1991,HubbardEM2004,RauchPLA1995}, a spatially extended system of three variables  with an additional term corresponding to population growth is considered: polarization  ${\bf p}$,  the bacterial density $\rho$ and the self-secreted chemical density $c$. The reproduction and death of bacteria are modeled by considering a classic logistic growth term, which is governed by the following partial differential equations as in \cite{MP2018,Liebchen2015,Liebchen2017}: \[
\left\{\begin{split}
&(\up{a})\quad\f{\p \rho}{\p t}=-\nabla\cdot(\rho v_{0}\mathbf{p})+D_{\rho}\nabla^{2}\rho+\alpha \rho(1-\rho/K),\quad \textrm{in}\quad \Omega,\\
&(\up{b})\quad\f{\p c}{\p t}=D_{c}\nabla^2c+k_{0}\rho-k_{d}c+k_{a}\nabla\cdot(\rho\mathbf{p}),\quad \textrm{in}\quad \Omega,\\
&(\up{c})\quad\f{\p \mathbf{p}}{\p t}=-\gamma\mathbf{p}+D_{p}\nabla^2\mathbf{p}+\beta\nabla c-\gamma_2|\mathbf{p}|^2\mathbf{p},\quad \textrm{in}\quad \Omega,\\
\end{split}\right.
\]
where $\Omega\subset R^{2}$, is a bounded rectangular domain; 
 $v_{0}$ is a constant and denotes the self-propulsion speed of the bacteria; $D_{\rho}$ is the diffusion coefficient, $\alpha$ stands for the growth rate, $K$ is the carrying capacity of the bacteria; $\beta$ indicates the chemotactic strength where positive $\beta$  represents chemoattraction, while negative $\beta$ represents chemorepulsion. $\gamma$ is the decay rate of $\mathbf{p}$ and $D_{p}$ is the translational diffusion constant.  $\gamma_{2}|\mathbf{p}|^2\mathbf{p}$  ensures the saturation of polarization at strong alignment. At a local rate $k_{0}$ the chemical substance is produced from bacteria  and naturally degraded at a rate $k_{d}$. The term $k_{a}\nabla\cdot(\rho\mathbf{p})$ describes an anisotropic correction to the isotropic secretion term $k_{0}\rho$.

Set $\tilde{t}=k_{d}t$, , $\tilde{\rho}=\rho /K$,  and $\tilde{\mathbf{p}}=v_{0}/\sqrt{k_{d}D_{\rho}}\mathbf{p}$, $\tilde{c}=(k_{d}/k_{0}K)c$. So we can rewrite the above system into an equivalent form in dimensionless quantities as follows
\begin{equation}\label{eq1}
\left\{\begin{split}
&(\up{a})\quad\f{\p \rho}{\p t}=-\nabla\cdot(\rho \mathbf{p})+\nabla^{2}\rho+g\rho(1-\rho),\quad \textrm{in}\quad \Omega,\\
&(\up{b})\quad\f{\p c}{\p t}=\mathcal{D}_{c}\nabla^2c+\rho -c+k\nabla\cdot(\rho\mathbf{p}),\quad \textrm{in}\quad \Omega,\\
&(\up{c})\quad\f{\p \mathbf{p}}{\p t}=-\Gamma\mathbf{p}+\mathcal{D}_{p}\nabla^2\mathbf{p}+s\nabla c-\Gamma_2|\mathbf{p}|^2\mathbf{p},\quad \textrm{in}\quad \Omega,\\
\end{split}\right.
\end{equation}
where $\Gamma=\gamma/k_{d}$, $\mathcal{D}_{p}=D_{p}/D_{\rho}$, $k=k_{a}k_{d}/(k_{0}v_{0})$, $g=\alpha/k_{d}$ and $\Gamma_{2}=\gamma_{2}D_{\rho}/v^{2}_{0}$, the initial conditions are given by
\begin{equation}\label{eq2}
 \rho(x,0)=\rho^0(x),\quad  c(x,0)=c^{0}(x), \quad\mathbf{p}(x,t)=\mathbf{p}^0(x)
\end{equation}
and the periodic boundary condition is considered.

In our knowledge, there are few research works on numerical analysis of growth-mediated autochemotactic pattern formation in self-propelling bacteria.  There are several interesting properties about the system,  one is that  many pattern formation will emerge before arriving at the steady state; the other is that, with the dynamic evolution process driven by the hydrolysis of ATP, the total free energy of the system probably increases or decreases. Therefore, many effective numerical schemes which keep energy unconditional stability, such as in \cite{Zhao2016, Zhao2017, Yang2017, Shen2015, Gong2020, Guo2017,Jiang2022,Jiang2021}, could not be used. As we know,  for the bacteria density equation, it is a parabolic-type equation and nonlinear convection-reaction-diffusion equation. Although there are many literature to deal with the  convection-dominated \cite{Bell1988, Johnson1986, YANG2000,dou,zhang2016} or  reaction-diffusion problem, focusing on this model there are few work. 
To obtain a better approximation, we propose a new modified characteristic Galerkin finite element method for the bacterial density equation  as in \cite{rui2010, Jiang2019,Jiang2021}. Meanwhile,  the classical Galerkin finite element method could be exploited to approximate the average polarization  and chemical density. As we pointed out, the new combined characteristic finite element method takes full advantages of the finite element methods and keeps the mass balance. Theoretically, we considered the convergence analysis, and derive the corresponding error estimate. Numerically, we also provide some numerical experiments to verify our theoretical results, and simulate the growth-mediated autochemotactic pattern formation in self-propelling bacteria.

We organize this paper as follows. Firstly, we will establish our new characteristic finite element method for the growth-mediated autochemotactic pattern formation in self-propelling bacteria in Section 2. Secondly, we consider the convergence  of the novel method, and derive the corresponding error estimate in Section 3.  And then, we find several wave pattern formation in the chemorepulsion regimes in Section 4.  At last, we draw some conclusions in Section 5.

\section{The formulation of numerical method}
\setcounter{equation}{0}
In this section, we will give the novel characteristic finite element method.  For this purpose, we denote the characteristic direction by $\tau$ such that
\[
\phi\f{\p }{\p\tau}=\f{\p }{\p t}+\mathbf{p}\cdot\nabla,\quad \phi=\sqrt{1+|\mathbf{p}|^2}.
\]
So  \eqref{eq1}  can be rewritten into the following equivalent form
 \begin{equation}\label{eq3}
\begin{split}
&(\up{a})\quad\phi\f{\p \rho}{\p \tau}=-\nabla\cdot\mathbf{p}\rho+\nabla^{2}\rho+g\rho(1-\rho),\\
&(\up{b})\quad \f{\p c}{\p t}=\mathcal{D}_{c}\nabla^2c+\rho -c+k\nabla\cdot(\rho\mathbf{p}),\\
&(\up{c})\quad\f{\p \mathbf{p}}{\p t}=\mathcal{D}_{p}\nabla^{2}\mathbf{p}-\Gamma \mathbf{p}-\Gamma_2|\mathbf{p}|^2\mathbf{p}+s\nabla c.
\end{split}
\end{equation}

To derive the time discrete scheme of system \eqref{eq1}, we set $N$ some positive integer, denote time increment $\tr t=T/N$. Define a uniform time partition: $0=: t_0<t_1<\cdots <t_n=n\tr t<\cdots <t_{N-1}<t_{N}:=T$. As in \cite{dou}, we know that  we can use the following formula to approximate the characteristic derivative
\[
\phi\f{\p \rho}{\p\tau}\approx \sqrt{1+|\mathbf{p}^{n-1}|^2}\frac{\rho^{n}- \tilde{\rho}^{n-1}}{\sqrt{(X^{n}-x)^{2}+(\tr t)^{2}}}
=\frac{ \rho^{n}-\tilde{\rho}^{n-1}}{\tr t}
\]
where $\tilde{\rho}^{n-1}=\rho(X^{n})$, $X^n=x-\mathbf{p}^{n-1}\tr t$.  Hence we get the  semi-discrete formulation in time of \eqref{eq1}(a) as follows
 \begin{equation}\label{eq4:1}
 \f{\rho^n-\tilde{\rho}^{n-1}}{\tr t}+\nabla\cdot\mathbf{p}^{n-1}\rho^{n}-\nabla^{2}\rho^n-g\rho^{n}(1-\rho^{n-1})=0.
\end{equation}
By modifying the characteristic approximation as in \cite{rui2010, Jiang2019}, 
\[
\frac{\rho^{n}- \tilde{\rho}^{n-1}\delta^{n}}{\tr  t}
\approx
\f{\p \rho }{\p t}+\nabla\cdot(\mathbf{p}\rho),
\quad\delta^{n}=\det(\f{\p X^{n}}{\p x}),
\]
 we  construct a new  semi-discrete formulation in time
 \begin{equation}\label{eq7}
\f{\rho^n-\tilde{\rho}^{n-1}\delta^n}{\tr t}-\nabla^{2}\rho^n=g\rho^{n}(1-\rho^{n-1}).
\end{equation}
\begin{thm}[Mass balance]
The discrete scheme \eqref{eq7} keeps mass balance.
\end{thm}
\begin{proof}
Integrating \eqref{eq1}(a) over $t$, we can obtain the semi-discrete mass balance equation
\begin{equation}\label{eq5}
\int_{\Omega}\rho^ndx=\int_{\Omega}\rho^{n-1}dx+\int_{\Omega}g\rho^{n}(1-\rho^{n-1})dx.
\end{equation}
Integrating \eqref{eq7} with respect to $x$ on $\Omega$, we have
\begin{equation}\label{eq8}
\int_{\Omega}\rho^ndx=\int_{\Omega}\tilde{\rho}^{n-1}\delta^ndx+\int_{\Omega}g\rho^{n}(1-\rho^{n-1})dx.
\end{equation}
Using  the periodic assumption and the inverse transformation, we know that
\[
\int_{\Omega}\tilde{\rho}^{n-1}\delta^ndx = \int_{X^{n}(\Omega)}\rho^{n-1}(y)\delta^{n}(\delta^{n})^{-1}dy=\int_{\Omega}\rho^{n-1}dx,
\]
this equation suggests that \eqref{eq7} keeps mass balance.
\end{proof}

Based on \eqref{eq7}, and time backward Euler difference scheme in time, we can get the following weak variational form 
 \begin{equation}\label{eq9}
\begin{split}
&(\up{a})\quad (\f{\rho^n-\tilde{\rho}^{n-1}\delta^{n}}{\tr t},v)+(\nabla\rho^n,\nabla v)=(g\rho^{n}(1-\rho^{n-1}),v) \quad v\in H^{1}(\Omega),\\
&(\up{b})\quad (\f{c^{n}-c^{n-1}}{\tr t},z)+(\mathcal{D}_{c}\nabla c^n,\nabla z)+(c^n,z)=(\rho^{n}+k\nabla\cdot(\rho^n\mathbf{p}^n),z) \quad z\in H^{1}(\Omega),\\
&(\up{c})\quad(\f{ \mathbf{p}^n-\mathbf{p}^{n-1}}{\tr t},\mathbf{q})+(\mathcal{D}_{p}\nabla\mathbf{p}^n,\nabla\mathbf{q})+([\Gamma +\Gamma_2|\mathbf{p}^{n-1}|^2]\mathbf{p}^n,\mathbf{q})=(s\nabla c^{n},\mathbf{q}) \quad\mathbf{q}\in [H^{1}(\Omega)]^{2}.
\end{split}
\end{equation}

Denote the uniform regular  partition of the domain $\Omega$ by $\mathcal{T}_{h}$, where the elements in the partition have the diameters bounded by $h$. And then, define two $r$-degree piecewise polynomial finite element spaces ${\mathcal W}_{h}\subset [H^1(\Omega)]^2$ and ${\mathcal V}_{h}\subset H^1(\Omega)$. Based on the weak variational form \eqref{eq9}, we propose the linear, decoupled characteristic Galerkin method for model problem \eqref{eq1}.

 \subsection*{\bf CFE Algorithm}
 Given initial condition $(\rho^0_h,\mathbf{p}^0_h,c^0_{h})\in\mathcal{V}_h\times\mathcal{W}_h\times\mathcal{V}_h$,  for $n=1,2,\ldots, N$, seek $(\rho^n_h,\mathbf{p}^n_h, c^{n}_{h})\in\mathcal{V}_h\times\mathcal{W}_h\times\mathcal{V}_h$, such that
\begin{equation}\label{eq12}
\begin{split}
&(\up{a})\quad(\f{\rho^n_{h}-\tilde{\rho}^{n-1}_{h}\delta^n_{h}}{\tr t},v_{h})+(\nabla\rho^n_{h},\nabla v_{h})=(g\rho^{n}_{h}(1-\rho^{n-1}_{h}),v_{h})\quad \forall v_{h} \in \mathcal{V}_h,\\
&(\up{b})\quad (\f{c^{n}_{h}-c^{n-1}_{h}}{\tr t},z_{h})+(\mathcal{D}_{c}\nabla c^n_{h},\nabla z_{h})+(c^n_{h},z_{h})=(\rho^{n}_{h}+k\nabla\cdot(\rho^n_{h}\mathbf{p}^{n-1}_{h}),z_{h})\quad \forall z_{h} \in \mathcal{V}_h,\\
&(\up{c})\quad(\f{ \mathbf{p}^n_{h}-\mathbf{p}^{n-1}_{h}}{\tr t},\mathbf{q}_{h})+(\mathcal{D}_{p}\nabla\mathbf{p}^n_{h},\nabla\mathbf{q}_{h})+([\Gamma +\Gamma_2|\mathbf{p}^{n-1}_{h}|^2]\mathbf{p}^n_{h},\mathbf{q}_{h})=(s\nabla c^{n}_{h},\mathbf{q}_{h})\quad \mathbf{q}_{h}\in\mathcal{W}_h,
\end{split}
\end{equation}
where $\tilde{\rho}^{n-1}_{h}=\rho(X^{n}_{h}), X^n_h=x-\mathbf{p}^{n-1}_h\tr t$,  $\delta^{n}_h=\det(\p X^{n}_h/\p x)$.

For CFE Algorithm, we can get the main convergence theorem as follows:
\begin{thm}\label{thm1}
Let $(\rho^n_h,\mathbf{p}^n_h,c^{n}_{h})$ be the solution of  CFE Algorithm. If the regularities of  the solution $(\rho^n,\mathbf{p}^n,c^{n})$ of the model problem \eqref{eq1}: $
\rho\in l^\infty(W^{1,\infty}), \rho_t\in l^2(H^{r+1}),\rho_{tt}\in l^2(L^{2}),\mathbf{p}\in l^\infty(W^{1,\infty}), \mathbf{p}_t\in l^2(H^{r+1}),\mathbf{p}_{tt}\in l^2(L^2)$, $
c\in l^\infty(W^{1,\infty}), c_t\in l^2(H^{r+1}),c_{tt}\in l^2(L^{2})$ hold,  the following error estimate holds
\begin{equation}\label{eq22}
\max_{n}\|\rho^n-\rho^n_h\|_{L^2}+\max_{n}\|\mathbf{p}^n-\mathbf{p}^n_h\|_{L^2}+\max_{n}\|\bm\sigma^n-\bm\sigma^n_h\|_{L^2}\leq C\{h^{r+1}+\tr t\}.
\end{equation}
\end{thm}

\section{Convergence analysis}
\setcounter{equation}{0}

In order to prove our convergence theorem, as in \cite{PGC}, we introduce two elliptic projection operators $\Pi_h$ and $\Pi_p$ such that
\begin{equation}\label{eq13}
(\nabla w,\nabla v_h)=(\nabla \Pi_hw, \nabla v_h),\quad \forall v_h\in \mathcal{V}_h,\quad (\nabla \mathbf{p},\nabla  \mathbf{q}_h)=(\nabla \Pi_p \mathbf{p},\nabla \mathbf{q}_h),\quad \forall \mathbf{q}_h\in \mathcal{W}_h.
\end{equation}
The following approximate properties hold
\begin{equation}\label{eq14}
\begin{split}
&(\textup{a})\quad\| w-\Pi_{h} w\|_{L^2}+h\|\nabla(w-\Pi_{h}w)\|_{[L^2]^2}\leq Ch^{r+1}\| w\|_{H^{r+1}},\quad \forall w\in H^{r+1}(\Omega),\\
&(\textup{b})\quad\|\mathbf{p}-\Pi_p \mathbf{p}\|_{[L^2]^2}+h\|\nabla(  \mathbf{p}-\Pi_p \mathbf{p})\|_{[L^2]^2}\leq Ch^{r+1}\|\mathbf{p}\|_{[H^{r+1}]^2},\quad \forall \mathbf{p}\in \in [H^{r+1}(\Omega)]^{2}.
\end{split}
\end{equation}
And the estimates can be easily obtained as follows:
\begin{equation}\label{eq14-1}
\begin{split}
&(\textup{a})\quad\| (w-\Pi_{h} w)_{t}\|_{L^2}\leq Ch^{r+1}\| w_{t}\|_{H^{r+1}},\quad \forall w\in H^{r+1}(\Omega),\\
&(\textup{b})\quad\|(\mathbf{p}-\Pi_p \mathbf{p})_{t}\|_{[L^2]^2}\leq Ch^{r+1}\|\mathbf{p}_{t}\|_{[H^{r+1}]^2},\quad \forall \mathbf{p}\in \in [H^{r+1}(\Omega)]^{2}.
\end{split}
\end{equation}

Set 
\[
\rho^n_h-\rho^n=(\rho^n_h-\Pi_{h}\rho^n)-(\rho^n-\Pi_{h}\rho^n)=\theta^n_{\rho}-\zeta^n_{\rho},
\]
\[
\mathbf{p}^n_h-\mathbf{p}^n=(\mathbf{p}^n_h-\Pi_p\mathbf{p}^n)-(\mathbf{p}^n-\Pi_p\mathbf{p}^n)=\theta^n_{p}-\zeta^n_{p}, 
\]
\[c^{n}_{h}-c^{n}=(c^{n}_{h}-\Pi_{h}c^{n})-(c^{n}-\Pi_{h}c^{n})=\theta_{c}^{n}-\zeta^{n}_{c}.
\]
Subtracting \eqref{eq3} from \eqref{eq12}, we can get the following error residual equations
\begin{equation}\label{eq18}
\begin{split}
&(\f{\theta^n_{\rho}-\tilde{\theta}^{n-1}_{\rho}\delta^n_h}{\tr t},v_h)+(\nabla\theta^n_{\rho},\nabla v_h)
\\=&(\f{\zeta^n_{\rho}-\tilde{\zeta}^{n-1}_{\rho}\delta^n_h}{\tr t},v_h)+(\f{\tilde{\rho}^{n-1}_h(\delta^n_h-\delta^n)}{\tr t},v_h)
+(g\rho^{n}_{h}(1-\rho^{n-1}_{h})-g\rho^{n}(1-\rho^{n}),v_{h})
\\&+(\f{(\tilde{\rho}^{n-1}_h-\tilde{\rho}^{n-1})\delta^n}{\tr t},v_h)
+(\f{\p \rho^n}{\p t}-\f{\rho^n-\tilde{\rho}^{n-1}\delta^n}{\tr t},v_h),
\end{split}
\end{equation}
\begin{equation}\label{eq21}
\begin{split}
(\f{\theta^{n}_{c}-\theta^{n-1}_{c}}{\tr t},z_{h})+(\mathcal{D}_{c}\nabla\theta_{c}^n,\nabla z_h)+(\theta_{c}^{n},z_{h})
=&(\f{\zeta^{n}_{c}-\zeta^{n-1}_{c}}{\tr t},z_{h})+(\f{\p c^n}{\p t}-\f{c^n-c^{n-1}}{\tr t},z_h)+(\zeta_{c}^{n},z_{h})
\\&
+(\rho^{n}_{h}-\rho^{n},z_{h})+(k\nabla\cdot(\rho^n_{h}\mathbf{p}^{n-1}_{h}-\rho^n\mathbf{p}^n),z_{h}),
\end{split}
\end{equation}
and
\begin{equation}\label{eq19}
\begin{split}
(\f{ \theta_{p}^n-\theta_{p}^{n-1}}{\tr t},\mathbf{q}_h)+(\mathcal{D}_{p}\nabla\theta_{p}^n,\nabla\mathbf{q}_h)
=&(\f{ \zeta_{p}^n-\zeta_{p}^{n-1}}{\tr t},\mathbf{q}_h)+(\f{ \p\mathbf{p}^n}{\p t}-\f{ \mathbf{p}^n-\mathbf{p}^{n-1}}{\tr t},\mathbf{q}_h)-(\Gamma[\mathbf{p}^n_h-\mathbf{p}^n],\mathbf{q}_h)
\\&-(\Gamma_2[|{\bf{P}}^{n-1}_h|^2\mathbf{p}^n_h-|{\bf{P}}^{n}|^2\mathbf{p}^n],\mathbf{q}_h)+(s\nabla (c^{n}_{h}-c^{n}),\mathbf{q}_{h}).
\end{split}
\end{equation}

For the proof of Theorem \ref{thm1}, the inductive hypothesis is necessary
\begin{equation}\label{eq23}
\max_n\|\mathbf{p}^n_h\|_{L^{\infty}}+\max_n\|\rho^n_h\|_{L^{\infty}}\leq C<+\infty.
\end{equation}
Now we estimate the boundedness of $\theta^n_{\rho}$, $\theta^n_{p}$ and $\theta^n_{c}$ one by one when the hypothesis \eqref{eq23} holds until $k=n-1$. 

\subsection{The estimate of $\theta^n_{p}$}
\begin{lem}\label{lem31}
For $\theta^n_{p}$, we have the estimate
\begin{equation}\label{eq25-1}
\begin{split}
\|\theta^n_{p}\|^2_{L^2}+\tr t\sum^n_{i=1}\|\nabla\theta_{p}^i\|^2_{L^2}
\leq &C\{\tr t\sum^n_{i=1}[\|\theta_{p}^i\|^2_{L^2}+\|\theta^{i}_{c}\|^2_{L^2}]+h^{2r+2}+\tr t^{2}\}.
\end{split}
\end{equation}
\end{lem}
\begin{proof}
Choosing $\mathbf{q}_h=\theta^n_{p}$ in \eqref{eq19}  we have
\begin{equation}\label{eq24-1}
\begin{split}
&\f{1 }{2\tr t}[(\theta^n_{p},\theta^n_{p})-(\theta_{p}^{n-1},\theta_{p}^{n-1})]+(\mathcal{D}_{p}\nabla\theta^n_{p},\nabla\theta^n_{p})
\\\leq&(\f{1}{\tr t}\int^{t_n}_{t_{n-1}}\f{ \p\zeta_{p}}{\p t}dt,\theta^n_{p})+(\f{ \p\mathbf{p}^n}{\p t}-\f{ \mathbf{p}^n-\mathbf{p}^{n-1}}{\tr t},\theta^n_{p})
\\&-(\Gamma[\mathbf{p}^n_h-\mathbf{p}^n]+\Gamma_2[|\mathbf{p}^{n-1}_h|^2\mathbf{p}^n_h-|\mathbf{p}^{n}|^2\mathbf{p}^n],\theta^n_{p})\\&+(s\nabla (c^{n}_{h}-c^{n}),\theta^n_{p})
=E_{1}+E_{2}+E_{3}+E_{4}.
\end{split}
\end{equation}

%%%%%%%%%%

Note that
\[
E_1+E_{2}\leq C\{\f{1}{\tr t}\int^{t_n}_{t_{n-1}}\|\f{ \p\zeta_{p}}{\p t}\|^2_{L^2}dt+\tr t\int^{t_n}_{t_{n-1}}\|\mathbf{p}_{tt}\|^2_{L^2}dt+\|\theta_{p}^n\|^2_{L^2}\}
\]
and
\[
E_{4}=-(s(c^{n}_{h}-c^{n}),\nabla\cdot\theta^n_{p})
\leq C\{\|\theta_{c}^n\|^2_{L^2}+\|\zeta_{c}^n\|^2_{L^2}\}\}+\f{1}{2}\|\mathcal{D}^{1/2}_{p}\nabla\theta^n_{p}\|^2_{L^2}.
\]

Then, we can give the estimate of $E_3$. We know that
\[
\begin{split}
E_3=&-(\Gamma[\mathbf{p}^n_h-\mathbf{p}^n],\theta_{p}^n)-(\Gamma_2[|\mathbf{p}^{n-1}_h|^2\mathbf{p}^n_h-|\mathbf{p}^{n}|^2\mathbf{p}^n],\theta_{p}^n)
\\=&-(\Gamma [\mathbf{p}^n_h-\mathbf{p}^n],\theta_{p}^n)
-(\Gamma_2|\mathbf{p}^{n-1}_h|^2(\mathbf{p}^n_h-\mathbf{p}^{n}),\theta_{p}^n) 
\\&-(\Gamma_2(\mathbf{p}^{n-1}_h+\mathbf{p}^{n})\cdot\mathbf{p}^n(\mathbf{p}^{n-1}_h-\mathbf{p}^{n}),\theta_{p}^n)\\
\leq &C\{\|\theta_{p}^n\|^2_{L^2}+\|\zeta_{p}^n\|^2_{L^2}+\|\theta^{n-1}_{p}\|^2_{L^2}+\|\zeta_{p}^n\|^2_{L^2}
+\tr t\int^{t_n}_{t_{n-1}}\|\mathbf{p}_{t}\|^2_{L^2}dt\},
\end{split}
\]
where the inductive hypothesis \eqref{eq23} has been used in the last inequality.

Substituting these estimates into \eqref{eq24-1},  we get the estimate
\[
\begin{split}
&\f{1 }{2\tr t}[(\theta^n_{p},\theta^n_{p})-(\theta_{p}^{n-1},\theta_{p}^{n-1})]+\f{1 }{2}(\mathcal{D}_{p}\nabla\theta^n_{p},\nabla\theta^n_{p})
\\\leq&C\{\|\theta^n_{p}\|^2_{L^2}+\|\theta^{n-1}_{p}\|^2_{L^2}+\|\theta^{n}_{c}\|^2_{L^2}+\|\zeta^{n}_{p}\|^2_{L^2}+\|\zeta^{n}_{c}\|^2_{L^2}
\\&+\f{1}{\tr t}\int^{t_n}_{t_{n-1}}\|\f{ \p\zeta_{p}}{\p t}\|^2_{L^2}dt+\tr t\int^{t_n}_{t_{n-1}}(\|\mathbf{p}_t\|^2_{L^2}+\|\mathbf{p}_{tt}\|^2_{L^2})dt\}.
\end{split}
\]
Multiplying it by $2\tr t$, and summing it over $n$, then we can obtain
\begin{equation}\label{eq25}
\begin{split}
\|\theta^n_{p}\|^2_{L^2}+\tr t\sum^n_{i=1}\|\nabla\theta_{p}^i\|^2_{L^2}
\leq &C\{\tr t\sum^n_{i=1}[\|\theta_{p}^i\|^2_{L^2}+\|\zeta_{p}^i\|^2_{L^2}+\|\theta^{i}_{c}\|^2_{L^2}+\|\zeta^{i}_{c}\|^2_{L^2}]
\\&+\|\f{\p \zeta_{p}}{\p t}\|^2_{L^2(0,t_n;L^2)}+\tr t^2(\|\mathbf{p}_t\|^2_{L^2(0,t_n;L^2)}+\|\mathbf{p}_{tt}\|^2_{L^2(0,t_n;L^2)}\}.
\end{split}
\end{equation}
By use of the estimate \eqref{eq14}, we know that the estimate \eqref{eq25-1} holds.
\end{proof}

\subsection{The estimate of $\theta^n_{\rho}$}
\begin{lem}\label{lem32}
For $\theta^n_{\rho}$, we have the estimate
\begin{equation}\label{eq33-1}
\begin{split}
\|\theta^n_{\rho}\|^2_{L^2}+\tr t\sum^n_{i=1}\|\nabla\theta^i_{\rho}\|^2_{L^2}
\leq &C\{\tr t\sum^{n}_{i=1}[\|\theta^{i-1}_{\rho}\|^2_{L^2}+\|\theta^{i-1}_{p}\|^2_{L^2}]+h^{2r+2}+\tr t^2\}.
\end{split}
\end{equation}
\end{lem}
\begin{proof}
Taking $v_h=\theta^n_{\rho}$ in \eqref{eq18},  we can get
then we can get
\begin{equation}\label{eq26}
\begin{split}
&(\f{\theta^n_{\rho}-\tilde{\theta}^{n-1}_{\rho}\delta^n_h}{\tr t},\theta^n_{\rho})+(\nabla\theta^n_{\rho},\nabla \theta^n_{\rho})
\\=&(\f{\p \rho^n}{\p t}-\f{\rho^n-\tilde{\rho}^{n-1}\delta^n}{\tr t},\theta^n_{\rho})+(\f{(\tilde{\rho}^{n-1}_h-\tilde{\rho}^{n-1})\delta^n}{\tr t},\theta^n_{\rho})
+(\f{\tilde{\rho}^{n-1}_h(\delta^n_h-\delta^n)}{\tr t},\theta^n_{\rho})\\&+(\f{\zeta^n_{\rho}-\tilde{\zeta}_{\rho}^{n-1}\delta^n_h}{\tr t},\theta^n_{\rho})
+(g(1+\rho^{n-1}_{h})(\theta^{n}_{\rho}-\zeta^{n}_{\rho}),\theta^{n}_{\rho})\\&+(g\rho^{n}(\theta^{n-1}_{\rho}-\zeta^{n-1}_{\rho}),\theta^{n}_{\rho})-(g\rho^{n}\int^{t_{n}}_{t_{n-1}}\rho_tdt,\theta^{n}_{\rho})
=T_1+T_2+T_3+\cdots+T_7.
\end{split}
\end{equation}
Using  the same technique as in \cite{rui2010}, we can get the inequality
\[
|T_1+T_2|\leq C\{\tr t^2\|\rho\|^2_{C^2(L^2)\cap C^1(H^1)\cap C^2(H^2)}+\|\mathbf{p}^{n-1}_h-\mathbf{p}^{n-1}\|^2_{L^2}\}+\f{\varepsilon}{4} \|\theta^n_{\rho}\|^2_{L^2}.
\]
Utilizing the definition of $\delta^n$ and $\delta^n_h$, we know that
\[
\delta^n=1-\tr t\nabla\cdot\mathbf{p}^{n-1}+O(\tr t^2),\quad
\delta^n_h=1-\tr t\nabla\cdot\mathbf{p}^{n-1}_h+O(\tr t^2).
\]
So for $T_3$ we have
\[
|T_3|\leq C\|\mathbf{p}^{n-1}_h-\mathbf{p}^{n-1}\|^2_{L^2}+\f{\varepsilon}{4}  \|\nabla\theta^n_{\rho}\|^2_{L^2}.
\]

For $T_4$, we have
\begin{equation}\label{eq29}
\begin{split}
T_4=&(\f{\zeta^n_{\rho}-\zeta^{n-1}_{\rho}}{\tr t},\theta^n_{\rho})+(\f{\zeta^{n-1}_{\rho}-\tilde{\zeta}^{n-1}_{\rho}\delta^n_h}{\tr t},\theta^n_{\rho})\\
=&(\f{\zeta^n_{\rho}-\zeta^{n-1}_{\rho}}{\tr t},\theta^n_{\rho})+(\f{\tilde{\zeta}^{n-1}_{\rho}\delta^n_h}{\tr t},\tilde{\theta}^n)-(\f{\tilde{\zeta}^{n-1}_{\rho}\delta^n_h}{\tr t},\theta^n_{\rho})\\
=&(\f{\zeta^n_{\rho}-\zeta^{n-1}_{\rho}}{\tr t},\theta^n_{\rho})+(\tilde{\zeta}^{n-1}_{\rho}\delta^n_h,\f{\tilde{\theta}^n_{\rho}-\theta^n_{\rho}}{\tr t}).
\end{split}
\end{equation}
To estimate the bound of $T_{4}$, we consider the following transformation, 
\[
y=f_{\bar{z}}=x-\mathbf{p}^{n-1}_h\tr t\bar{z}.
\]
It is easily seen that
\[
\begin{split}
\|\f{\tilde{\theta}^n_{\rho}-\theta^n_{\rho}}{\tr t}(\delta^n_h)^{1/2}\|^2_{L^2}=&
\tr t^{-2}\int_\Omega(\theta^n_{\rho}-\tilde{\theta}^n_{\rho})^2\delta^n_hdx=
\tr t^{-2}\int_\Omega(\int^{x}_{X^n_h}\f{\p\theta^n_{\rho}}{\p z}dz)^2\delta^n_hdx\\
= &\tr t^{-2}\int_\Omega(\int^{1}_{0}\f{\p\theta^n_{\rho}}{\p z}(x-\mathbf{p}^{n-1}_h\tr t\bar{z})\cdot(\mathbf{p}^{n-1}_h\tr t\bar{z})d\bar{z})^2\delta^n_hdx\\
\leq &C\int^{1}_{0}\int_\Omega|\nabla\theta^n_{\rho}(x-\mathbf{p}^{n-1}_h\tr t\bar{z})|^2\delta^n_hdxd\bar{z}
\leq C\|\nabla\theta^n_{\rho}\|^2_{L^2}.
\end{split}
\]
Using the above inequality, we can get
\[
\begin{split}
|T_4|\leq &\|\f{\zeta^n_{\rho}-\zeta^{n-1}_{\rho}}{\tr t}\|_{L^2}\|\theta^n_{\rho}\|_{L^2}+\|\tilde{\zeta}^{n-1}_{\rho}(\delta^n_h)^{1/2}\|_{L^2}\|\f{\tilde{\theta}^n_{\rho}-\theta^n_{\rho}}{\tr t}(\delta^n_h)^{1/2}\|_{L^2}
\\
\leq& C\{\f{1}{\tr t}\int^{t_n}_{t_{n-1}}\|\f{\zeta_{\rho}}{\p t}\|^2_{L^2}dt+\|\zeta^{n-1}_{\rho}\|^2_{L^2}\}+\f{\varepsilon}{4}\|\nabla \theta^n_{\rho}\|^2_{L^2}.
\end{split}
\]

For $T_{5}$, $T_{6}$ and $T_{7}$, using Schwarz inequality, we can get
\[
\begin{split}
T_{5}+T_{6}+T_{7}\leq C\{\|\theta^{n}_{\rho}\|^2_{L^2}+\|\theta^{n-1}_{\rho}\|^2_{L^2}+\|\zeta^{n}_{\rho}\|^2_{L^2}+\|\zeta^{n-1}_{\rho}\|^2_{L^2}+\tr t\int^{t_{n}}_{t_{n-1}}\|\rho_{t}\|^{2}_{L^{2}}dt\}
\end{split}
\]

Substituting the above estimates into \eqref{eq26}, we get
\[
\begin{split}
&(\f{\theta^n_{\rho}-\tilde{\theta}^{n-1}_{\rho}\delta^n_h}{\tr t},\theta^n_{\rho})+(\nabla\theta^n_{\rho},\nabla \theta^n_{\rho})
\\\leq &C\{\|\theta^{n}_{\rho}\|^2_{L^2}+\|\theta^{n-1}_{\rho}\|^2_{L^2}+\|\theta^{n-1}_{p}\|^2_{L^2}+\|\zeta^{n-1}_{p}\|^2_{L^2}+\|\zeta^{n}_{\rho}\|^2_{L^2}+\|\zeta^{n-1}_{\rho}\|^2_{L^2}
\\
&+\tr t\int^{t_n}_{t_{n-1}}\|\rho_t\|^2_{L^2}dt
+\f{1}{\tr t}\int^{t_n}_{t_{n-1}}\|\f{\p \zeta_{\rho}}{\p t}\|^2_{L^2}dt
\\
&+\tr t^2\|\rho\|^2_{C^2(L^2)\cap C^1(H^1)\cap C^2(H^2)}\}
+\f{\varepsilon}{2} [\| \theta^n_{\rho}\|^2_{L^2}+\|\nabla \theta^n_{\rho}\|^2_{L^2}].
\end{split}
\]

It is easily seen that
\[
\begin{split}
(\f{\theta^n_{\rho}-\tilde{\theta}^{n-1}_{\rho}\delta^n_h}{\tr t},\theta^n_{\rho})&\geq \f{1}{2\tr t}[\|\theta^n_{\rho}\|^2_{L^2}-\|\tilde{\theta}^{n-1}_{\rho}\delta^n_h\|^2_{L^2}]
\\&\geq \f{1}{2\tr t}[\|\theta^n_{\rho}\|^2_{L^2}-(1+C\tr t)\|\theta^{n-1}_{\rho}\|^2_{L^2}].
\end{split}
\]
Hence,  for sufficiently small $\varepsilon$, we have
\[
\begin{split}
&\|\theta^n_{\rho}\|^2_{L^2}-\|\theta^{n-1}_{\rho}\|^2_{L^2}+\tr t\|\nabla\theta^n_{\rho}\|^2_{L^2}
\\\leq &C\tr t\{\|\theta^{n}_{\rho}\|^2_{L^2}+\|\theta^{n-1}_{\rho}\|^2_{L^2}+\|\theta^{n-1}_{p}\|^2_{L^2}+\|\zeta^{n-1}_{p}\|^2_{L^2}+\|\zeta^{n}_{\rho}\|^2_{L^2}+\|\zeta^{n-1}_{\rho}\|^2_{L^2}
+\tr t\int^{t_n}_{t_{n-1}}\|\rho_t\|^2_{L^2}dt
\\&+\f{1}{\tr t}\int^{t_n}_{t_{n-1}}\|\f{\p \zeta_{\rho}}{\p t}\|^2_{L^2}dt
+\tr t^2\|\rho\|^2_{C^2(L^2)\cap C^1(H^1)\cap C^2(H^2)}\}.
\end{split}
\]
Summing the estimate from $1$ to $n$, we get
\begin{equation}\label{eq33}
\begin{split}
&\|\theta^n_{\rho}\|^2_{L^2}+\tr t\sum^n_{i=1}\|\nabla\theta^i_{\rho}\|^2_{L^2}
\\\leq &C\{\tr t\sum^{n}_{i=1}[\|\theta^{i}_{\rho}\|^2_{L^2}+\|\zeta^{i}_{\rho}\|^2_{L^2}+\|\theta^{i-1}_{p}\|^2_{L^2}+\|\zeta^{i-1}_{p}\|^2_{L^2}]+\|\f{\p\zeta_{\rho}}{\p t}\|^2_{L^2(0,t_n;L^2)}\\
&+\tr t^2\|\rho_t\|^2_{L^2(0,t_n;L^2)}+\tr t^2\|\rho\|^2_{C^2(L^2)\cap C^1(H^1)\cap C^2(H^2)}\}.
\end{split}
\end{equation}
Using the esimates \eqref{eq14} and \eqref{eq14-1}, we get the inequality \eqref{eq33-1}.
\end{proof}

\subsection{The estimate of $\theta^n_{c}$}
\begin{lem}\label{lem33} For $\theta^n_{c}$, we have the estimate
\begin{equation}\label{eq36-1}
\begin{split}
\|\theta_c^{n}\|^{2}_{L^{2}}+\tr t\sum^{n}_{i=1}\|\nabla \theta^{i}_{c}\|^{2}_{L^{2}}
\leq C\{\tr t\sum^{n}_{i=1}[
\|\theta^{i-1}_{p}\|^2_{L^2}+\|\theta_\rho^{i}\|^{2}_{L^{2}}]+h^{2r+2}+\tr t^{2}
\}.
\end{split}
\end{equation}
\end{lem}
\begin{proof}
Choosing $z_{h}=\theta^{n}_{c}$ in \eqref{eq21}, we can get
\begin{equation}\label{eq34}
\begin{split}
&\f{1}{2\tr t}[\|\theta_c^{n}\|^{2}_{L^{2}}-\|\theta_c^{n-1}\|^{2}_{L^{2}}]+(\mathcal{D}_{c}\nabla\theta_{c}^n,\nabla \theta_{c}^n)+\|\theta_{c}^{n}|^{2}_{L^{2}}
\\
\leq&(\f{\zeta^{n}_{c}-\zeta^{n-1}_{c}}{\tr t},\theta_{c}^{n})+(\f{\p c^n}{\p t}-\f{c^n-c^{n-1}}{\tr t},\theta_{c}^{n})+(\zeta_{c}^{n},\theta_{c}^{n})
\\&
+(\rho^{n}_{h}-\rho^{n},\theta_{c}^{n})+(k\nabla\cdot(\rho^n_{h}\mathbf{p}^{n-1}_{h}-\rho^n\mathbf{p}^n),\theta_{c}^{n})
\\=&I_{1}+I_{2}+\cdots+I_{5}.
\end{split}
\end{equation}

For $I_{1}$, $I_{2}$ , $I_{3}$ and $I_{4}$ , we have the estimate
\[
\begin{split}
I_{1}+I_{2}+I_{3}+I_{4}
\leq C\{\int^{t_{n}}_{t_{n-1}}\|\f{\p \zeta_c}{\p t}\|^{2}_{L^{2}}dt+\tr t\int^{t_{n}}_{t_{n-1}}\|c_{tt}\|^{2}_{L^{2}}dt+\|\zeta_c^{n}\|^{2}_{L^{2}}+\|\theta_\rho^{n}\|^{2}_{L^{2}}+\|\zeta_\rho^{n}\|^{2}_{L^{2}}]\}+\f{1}{2}\|\theta_c^{n}\|^{2}_{L^{2}}.
\end{split}
\]
Note that
\[
\begin{split}
I_{5}=&-(k[(\rho^n_{h}-\rho^{n})\mathbf{p}^{n-1}_{h}],\nabla\theta_c^{n})-(k[\rho^n(\mathbf{p}^{n-1}_{h}-\mathbf{p}^{n-1})],\nabla\theta_c^{n})
-(k[\rho^n(\mathbf{p}^{n-1}-\mathbf{p}^n)],\nabla\theta_c^{n})
\\\leq &C\{\|\theta_\rho^{n}\|^{2}_{L^{2}}+\|\theta_p^{n-1}\|^{2}_{L^{2}}+\|\zeta_\rho^{n}\|^{2}_{L^{2}}+\|\zeta_p^{n-1}\|^{2}_{L^{2}}+\tr t\int^{t_{n}}_{t_{n-1}}\|\mathbf{p}_{t}\|^{2}_{L^{2}}dt\}+\f{1}{2}\tr t\|\mathcal{D}^{1/2}_{c}\nabla\theta^{n}_{c}\|^{2}_{L^{2}}.
\end{split}
\]

Substituting the above estimates into \eqref{eq34} and multiplying it by $2\tr t$, we get
\[
\begin{split}
&\|\theta_c^{n}\|^{2}_{L^{2}}-\|\theta_c^{n-1}\|^{2}_{L^{2}}+\tr t\|\mathcal{D}^{1/2}_{c}\nabla \theta^{n}_{c}\|^{2}_{L^{2}}\\
\leq &C\tr t\{\f{1}{\tr t}\int^{t_{n}}_{t_{n-1}}\|\f{\p \zeta_c}{\p t}\|^{2}_{L^{2}}dt+\|\zeta_c^{n}\|^{2}_{L^{2}}+\|\zeta_\rho^{n}\|^{2}_{L^{2}}+\|\zeta_p^{n-1}\|^{2}_{L^{2}}+\|\theta_\rho^{n}\|^{2}_{L^{2}}\\
&+\|\theta_p^{n-1}\|^{2}_{L^{2}}+\tr t\int^{t_{n}}_{t_{n-1}}[\|\mathbf{p}_{t}\|^{2}_{L^{2}}+\|c_{tt}\|^{2}_{L^{2}}]dt
\}.
\end{split}
\]
Hence,  utilizing  \eqref{eq14} and summing the above estimate from $1$  to $n$, we can obtain the inequality \eqref{eq36-1}.
\end{proof}

\subsection{The proof of Theorem \ref{thm1}}

\begin{proof}
Utlizing Lemmas \ref{lem31}-\ref{lem33}, we  have the estimate
\begin{equation}\label{eq41}
\begin{split}
&\|\theta_\rho^{n}\|^{2}_{L^{2}}+\|\theta_p^{n}\|^{2}_{L^{2}}+\|\theta_c^{n}\|^{2}_{L^{2}}+\tr t\sum^{n}_{i=1}[\|\nabla\theta^{i}_{\rho}\|^{2}_{L^{2}}+\|\nabla\theta^{i}_{p}\|^{2}_{L^{2}}+\|\nabla \theta^{i}_{c}\|^{2}_{L^{2}}]\\
\leq &C\{\tr t\sum^{n}_{i=1}[\|\theta^{i-1}_{p}\|^2_{L^2}+\|\theta^{i-1}_{c}\|^2_{L^2}+\|\theta^{i-1}_{\rho}\|^2_{L^2}]+h^{2r+2}+\tr t^{2}
\}.
\end{split}
\end{equation}
The discrete Gronwall's lemma results in the estimate
\begin{equation}\label{eq42}
\|\theta_\rho^{n}\|^{2}_{L^{2}}+\|\theta_p^{n}\|^{2}_{L^{2}}+\|\theta_\sigma^{n}\|^{2}_{L^{2}}+\tr t\sum^{n}_{i=1}[\|\nabla\theta^{i}_{\rho}\|^{2}_{L^{2}}+\|\nabla\theta^{i}_{p}\|^{2}_{L^{2}}+\|\nabla\cdot\theta^{i}_{\sigma}\|^{2}_{L^{2}}]
\leq C\{h^{2r+2}+\tr t^{2}\}
\end{equation}
Utilzing the approximate  properties \eqref{eq14}, we the convergence result of Theorem \ref{thm1}.

Finally we check the inductive hypothesis \eqref{eq23}.   From \eqref{eq42}, we can get the estiamte
\begin{equation}\label{eq31}
\begin{split}
\|\mathbf{p}^n_h\|_{L^\infty}+\|\rho^n_h\|_{L^\infty}\leq &Ch^{-1}\{\|\theta^n_{p}\|_{L^2}+\|\theta^n_{\rho}\|_{L^2}\}+(\|\zeta^n_{p}\|_{L^\infty}+\|\zeta^n_{\rho}\|_{L^\infty}+\|p^n\|_{L^\infty}+\|\rho^n\|_{L^\infty})
\\ \leq& Ch^{-1}(h^{r+1}+\tr t)+Ch^{r+1}+\|\mathbf{p}^n\|_{L^\infty}+\|\rho^n\|_{L^\infty}\leq C,
\end{split}
\end{equation}
So, the inductive hypothesis \eqref{eq23} holds.
\end{proof}

\section{Numerical examples}
%%%%%%%%%%%%%%%%%%%%%%%%%%%%%%%%%%%%%

\subsection{Convergence test}
Set $\Omega=[0,1]\times [0,1]$. We show several numerical results for the coupled problem \eqref{eq1} with our proposed method. The coefficients of the coupled system \eqref{eq1} can be chosen as follows: 
\[
D_c=1,\quad D_p=1, \quad s=0.5,\quad k=1,\quad \Gamma=1,\quad \Gamma_2=10,\quad g=0.1.
\]    
The exact solution of the coupled system \eqref{eq1} is taken by 
 \begin{equation}
\begin{split}
&\rho(x,y,t)=\sin(4\pi x)\sin(4\pi y)\exp(\sin(t))\\
&c(x,y,t)=\cos(4\pi x)\cos(4\pi y)\exp(\cos(t)),\\
&p_1(x,y,t)=\sin(4\pi x)\cos(4\pi y)\exp(\sin(t)),\\
&p_2(x,y,t)=\cos(4\pi x)\sin(4\pi y)\exp(\cos(t)).
\end{split}
\end{equation}
The piecewise linear polynomial space is considered. For different mesh size $h$ and $\Delta t=h^2$, some numerical results are presented in Table \ref{tab4}, Table \ref{tab5} and Table \ref{tab6}  for $\rho$, $c$ and $\|{\bf P}\|$, where  $\|{\bf P}\|^{2}=\int_{\Omega}(p_1^2+p_2^2)d{\bf{x}}$. From these tables, we can easily see that the new method is of time first-order accuracy and  spacial second-order accuracy both in $L^2$-norm and  $L^\infty$-norm, which is coincided with  theoretical result.

\begin{table}[H]
\begin{center}
\caption{Convergence results  for $\rho$}\label{tab4}
\begin{tabular}{c|cccc|cccc}
\hline
$h$
&\multicolumn{4}{c|}{Convergence rates in space}  &\multicolumn{4}{c}{Convergence rates in time }       \\ \cline{2-9}
&  $L^\infty$& rate & $L^2$ &rate& $L^\infty$ & rate & $L^2$ &rate
\\
\hline
$1/{8}$&9.61e-02&-&7.68e-02&-&9.61e-02&-&7.68e-02&-\\
$1/{16}$&2.41e-02&2.0&1.01e-02&2.0&2.41e-02&1.0&1.01e-02&1.0\\
$1/{32}$&6.02e-03&2.0&2.25e-03&2.0&6.02e-03&1.0&2.25e-03&1.0\\
$1/{64}$&1.52e-03&2.0&5.61e-04&2.0&1.52e-03&1.0&5.61e-04&1.0\\
\hline
\end{tabular}
\end{center}
\end{table}

\begin{table}[H]
\begin{center}
\caption{Convergence results for $c$}\label{tab5}
\begin{tabular}{c|cccc|cccc}
\hline
$h$
&\multicolumn{4}{c|}{Convergence rates in space}  &\multicolumn{4}{c}{Convergence rates in time}       \\ \cline{2-9}
&  $L^\infty$& rate & $L^2$ &rate& $L^\infty$ & rate & $L^2$ &rate
\\
\hline
$1/{8}$&2.13e-01&-&1.52e-01&-&2.13e-01&-&1.52e-01&-\\
$1/{16}$&5.22e-02&2.0&3.79e-02&2.0&5.22e-02&1.0&3.79e-02&1.0\\
$1/{32}$&1.28e-02&2.0&9.45e-03&2.0&1.28e-02&1.0&9.45e-03&1.0\\
$1/{64}$&3.01e-03&2.0&2.35e-03&2.0&3.01e-03&1.0&2.35e-03&1.0\\
\hline
\end{tabular}
\end{center}
\end{table}

\begin{table}[H]
\begin{center}
\caption{Convergence results for $\|{\bf P}\|$}\label{tab6}
\begin{tabular}{c|cccc|cccc}
\hline
$h$
&\multicolumn{4}{c|}{Convergence rates in space}  &\multicolumn{4}{c}{Convergence rates in time}       \\ \cline{2-9}
&  $L^\infty$& rate & $L^2$ &rate& $L^\infty$ & rate & $L^2$ &rate
\\
\hline
$1/{8}$&1.70e-01&-&1.28e-01&-&1.70e-01&-&1.28e-01&-\\
$1/{16}$&4.20e-02&2.0&3.18e-02&2.0&4.20e-02&1.0&3.18e-02&1.0\\
$1/{32}$&1.03e-02&2.0&7.93e-03&2.0&1.03e-02&1.0&7.93e-03&1.0\\
$1/{64}$&2.55e-03&2.0&1.95e-03&2.0&2.55e-03&1.0&1.95e-03&1.0\\
\hline
\end{tabular}
\end{center}
\end{table}

In order to validate the efficiency of our proposed method  in the following subsections, we will give some numerical results to simulate dynamics of the clustering and pattern formation in the repulsion case as listed in  \cite{Liebchen2015, Mukherjee2018}, meanwhile, we will consider some other value of parameter to try find the new wave pattern formation. Throughout the following section, the parameters are chosen as follows:
\[
D_c=1.0,\Gamma_2=10,\Gamma=1.0,D_p=1.0,k=0.5.
\]
The mesh size and time increment are $0.34$ and $0.01$, respectively. The period boundary condition is still considered.

\subsection{Chemorepulsion case 1}

In this case, the initial profile of bacterial density, self-chemical density and polarization given as $(u_{0},c_{0},{\bf{p}}_0)$. $u_{0}=0.1exp(-200*(x-0.5*Lx)^2-200*(y-0.5*Ly)^2)$,$c_0=u_{0}$, ${{\bf{p}}_0}=0.01(rand(0,1),rand(0,1))$, $Lx=60$, $Ly=60$, $d=0.01$, $h=0.6$. The evolution of the process is shown as  in Fig.\ref{RE_p1u}-\ref{RE_p1p}. These numerical results show that clustering and pattern formation appear with the increase of time, and finally both of them form one  special order. It seems that firstly the regular clustering appears and parallel to the boundary of domain as shown in Fig.\ref{RE_p1c}e and Fig.\ref{RE_p1u}f, secondly, the total clustering gathering into four parts as represented in Fig.\ref{RE_p1c}f and Fig.\ref{RE_p1u}h, thirdly, each part following the counter-diagonal direction move as Fig.\ref{RE_p1c}g, finally all the parts move in the same velocity and small clustering is back to one whole as Fig.\ref{RE_p1c}h. Then the pattern will get into the first step. The process of this cycle never stops. The system arrives at one equilibrium state. The Fig.\ref{RE_p1p} shows the dynamic process of polarization which clearly represent pattern formation and wave appears  association with bacteria and self-secreted chemical. 

\begin{figure}[H]
\centering
\subfigure[ $u$ at time 50 ]{\includegraphics[width=0.22\textwidth]{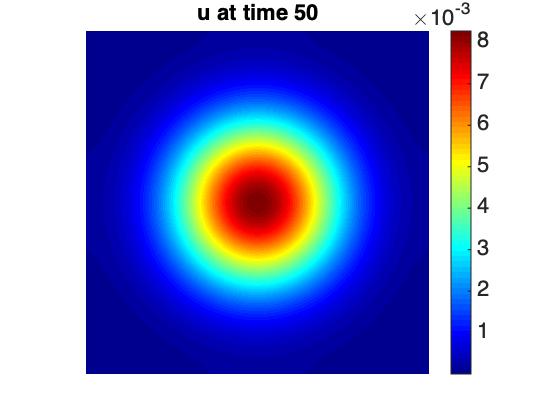}}
\subfigure[ $u$ at time 200 ]{\includegraphics[width=0.22\textwidth]{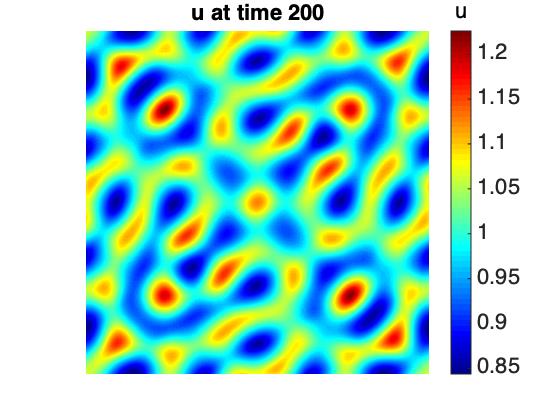}}
\subfigure[ $u$ at time 300 ]{\includegraphics[width=0.22\textwidth]{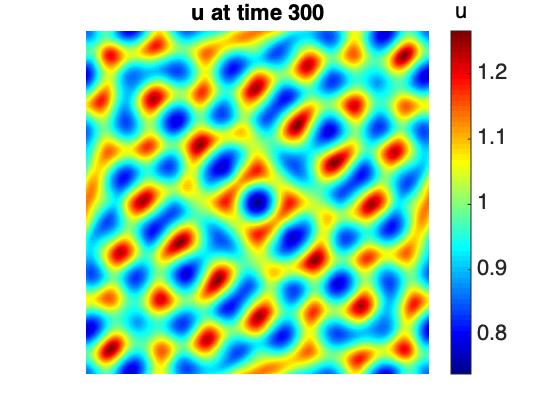}}
\subfigure[ $u$ at time 400 ]{\includegraphics[width=0.22\textwidth]{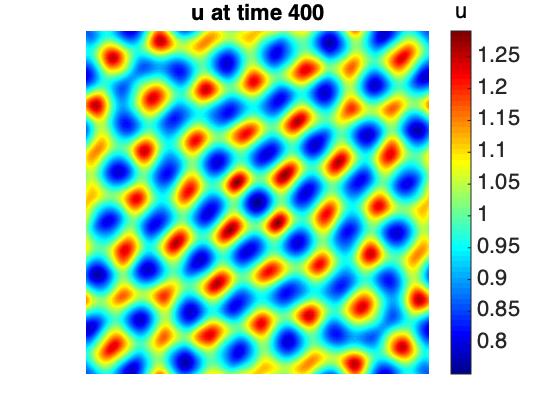}}\\

\subfigure[ $u$ at time 500 ]{\includegraphics[width=0.22\textwidth]{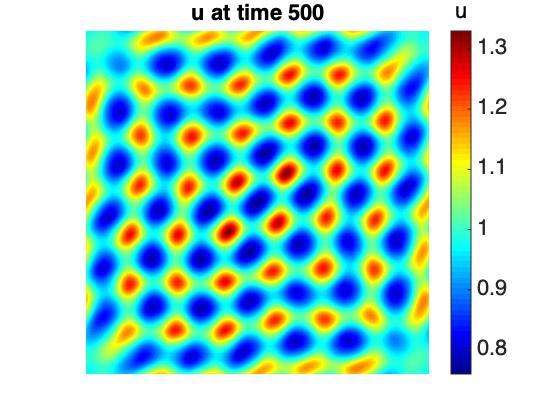}}
\subfigure[ $u$ at time 600 ]{\includegraphics[width=0.22\textwidth]{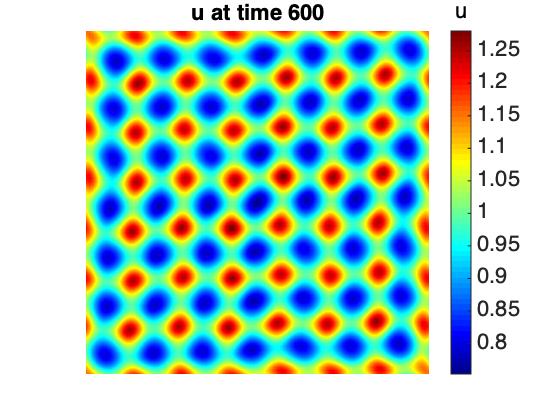}}
\subfigure[ $u$ at time 700 ]{\includegraphics[width=0.22\textwidth]{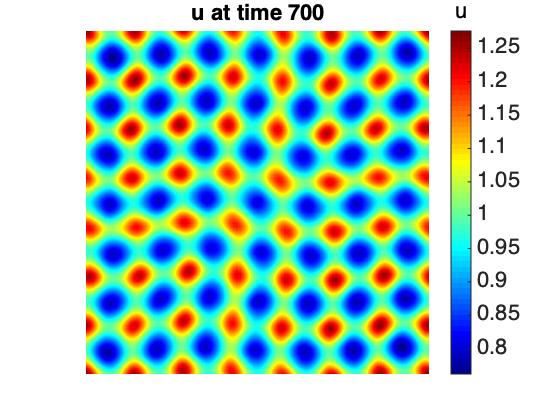}}
\subfigure[ $u$ at time 800 ]{\includegraphics[width=0.22\textwidth]{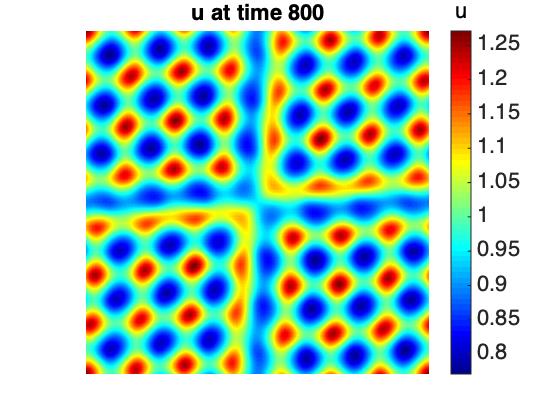}}\\

\caption{The dynamic process for bacterial density  with small initial distribution and $s=-15$, $g=0.1$.}
\label{RE_p1u}
\end{figure}

\begin{figure}[H]
\centering
\subfigure[$c$ at time 50]{\includegraphics[width=0.22\textwidth]{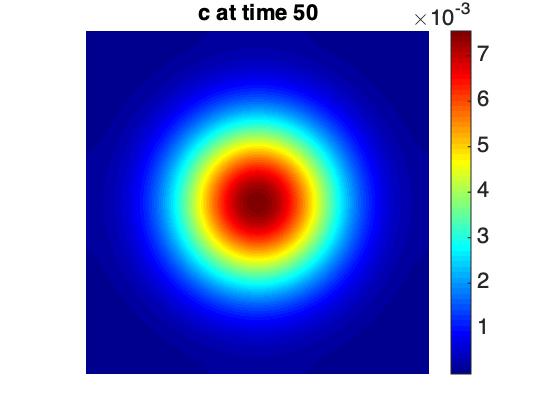}}
\subfigure[$c$ at time 200]{\includegraphics[width=0.22\textwidth]{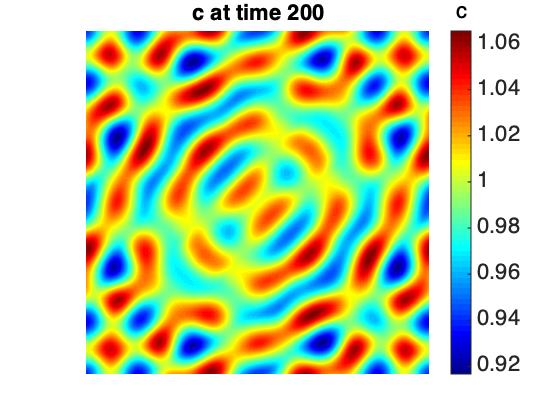}}
\subfigure[$c$ at time 300]{\includegraphics[width=0.22\textwidth]{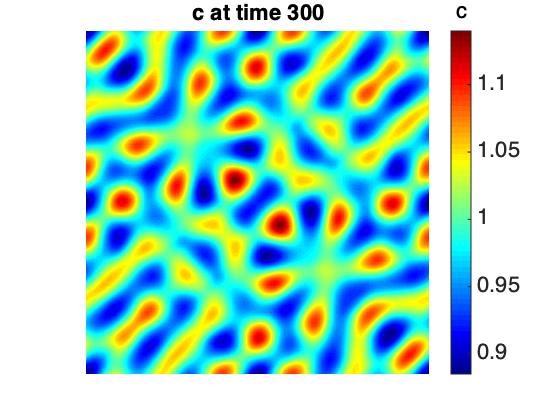}}
\subfigure[$c$ at time 400]{\includegraphics[width=0.22\textwidth]{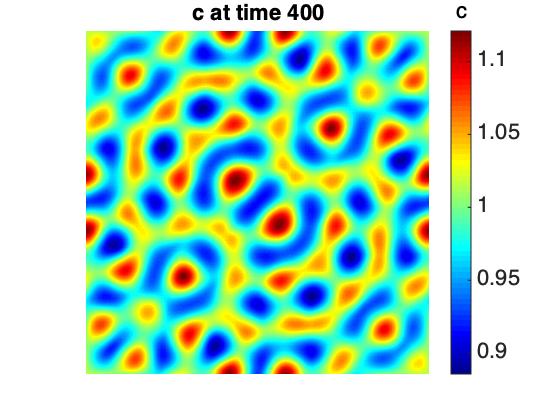}}\\

\subfigure[$c$ at time 500]{\includegraphics[width=0.22\textwidth]{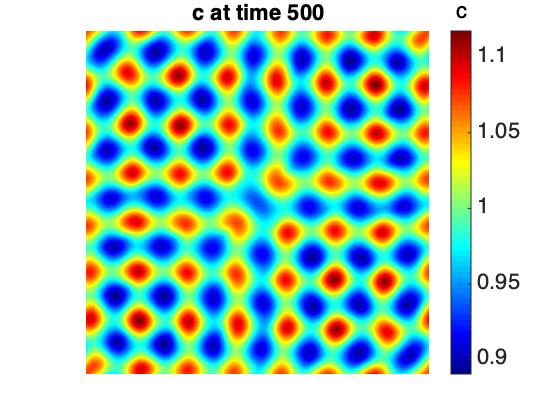}}
\subfigure[$c$ at time 600]{\includegraphics[width=0.22\textwidth]{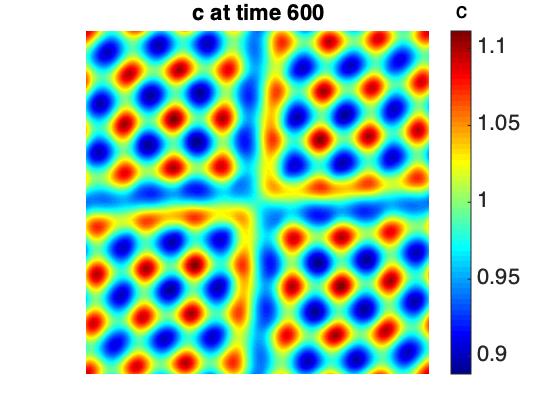}}
\subfigure[$c$ at time 700]{\includegraphics[width=0.22\textwidth]{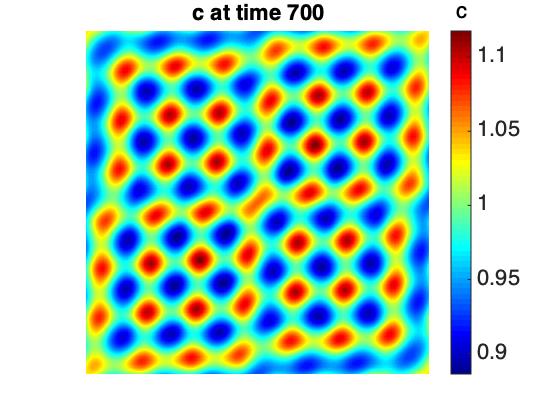}}
\subfigure[$c$ at time 800]{\includegraphics[width=0.22\textwidth]{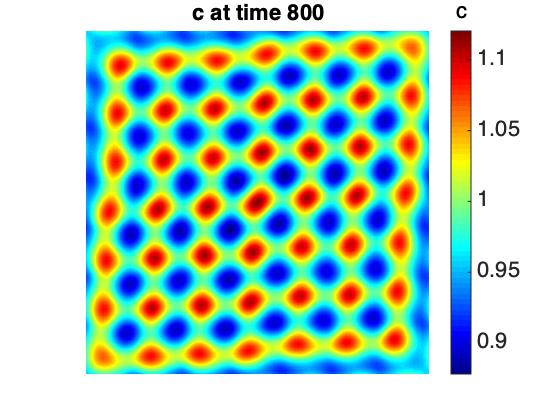}}\\
\caption{The dynamic process for density of self-secreted chemical  with small initial distribution and $s=-15$, $g=0.1$.}
\label{RE_p1c}
\end{figure}

\begin{figure}[H]
\centering
\subfigure[${\bf{p}}$ at time 50 ]{\includegraphics[width=0.22\textwidth]{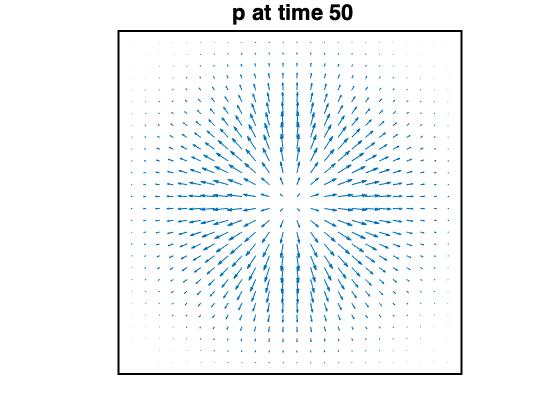}}
\subfigure[${\bf{p}}$ at time 200 ]{\includegraphics[width=0.22\textwidth]{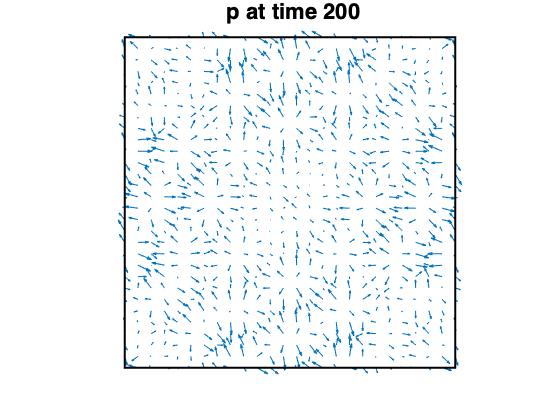}}
\subfigure[${\bf{p}}$ at time 300 ]{\includegraphics[width=0.22\textwidth]{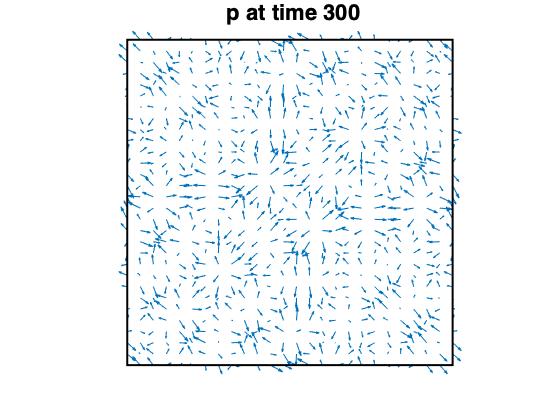}}
\subfigure[${\bf{p}}$ at time 400 ]{\includegraphics[width=0.22\textwidth]{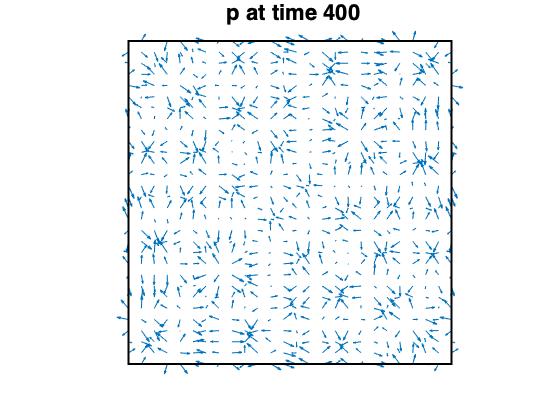}}\\

\subfigure[${\bf{p}}$ at time 500 ]{\includegraphics[width=0.22\textwidth]{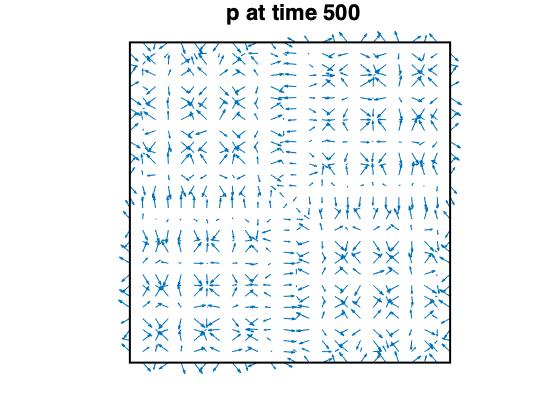}}
\subfigure[${\bf{p}}$ at time 600 ]{\includegraphics[width=0.22\textwidth]{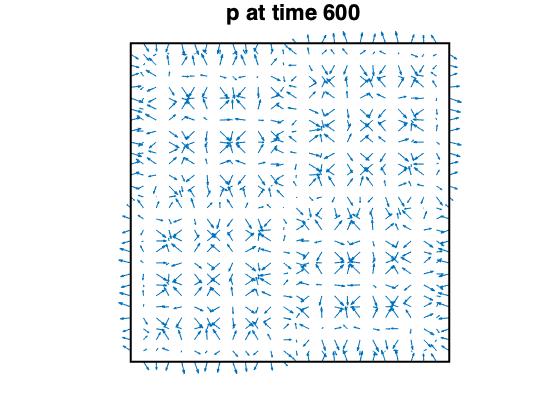}}
\subfigure[${\bf{p}}$ at time 700 ]{\includegraphics[width=0.22\textwidth]{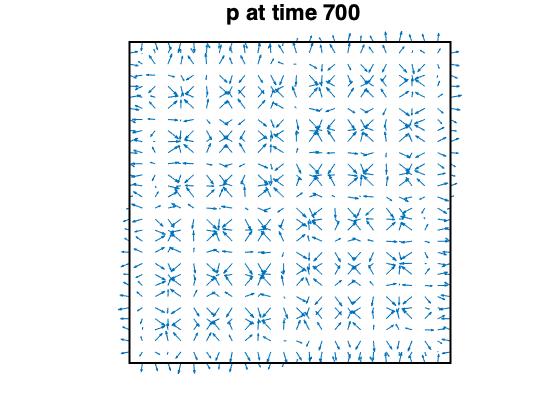}}
\subfigure[${\bf{p}}$ at time 800 ]{\includegraphics[width=0.22\textwidth]{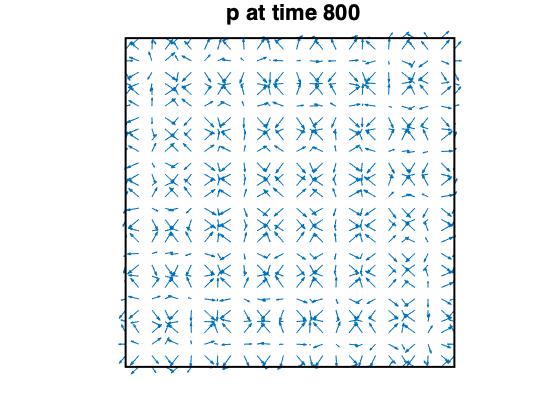}}\\

\caption{The dynamic process for density of polarization with small initial distribution and $s=-15$, $g=0.1$.}
\label{RE_p1p}
\end{figure}

\subsection{Chemorepulsion case 2}

Here the initial condition is chosen with a small perturbation of the uniform state $(1,1,{\bf{0}})$. The value of other parameters is the same as chemorepulsion case 1. The evolution of the process is presented as in Fig.\ref{RE_p1_1u}-\ref{RE_p1_1p}. We can find that clustering, pattern formation and wave appear. Comparing with case 1, the only different is the initial distribution. We could see that if the large number of  bacteria live in the environment with enough self-secreted chemicals. The wave also appear just the intermediate pattern formation different. Firstly clustering appear and then it forms one line which moves along a direction with an inclination to the boundary of domain. And finally the system reaches the equilibrium state.

\begin{figure}[H]
\centering
\subfigure[ $u$ at time 50 ]{\includegraphics[width=0.22\textwidth]{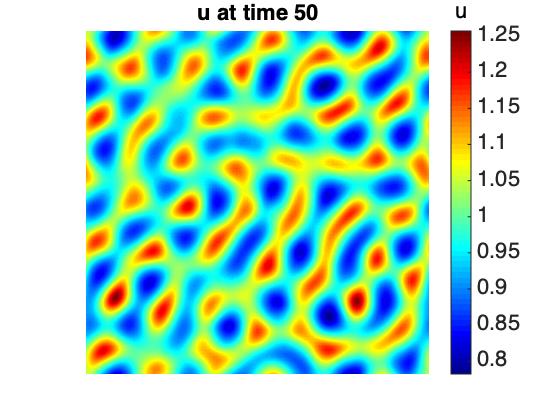}}
\subfigure[ $u$ at time 200 ]{\includegraphics[width=0.22\textwidth]{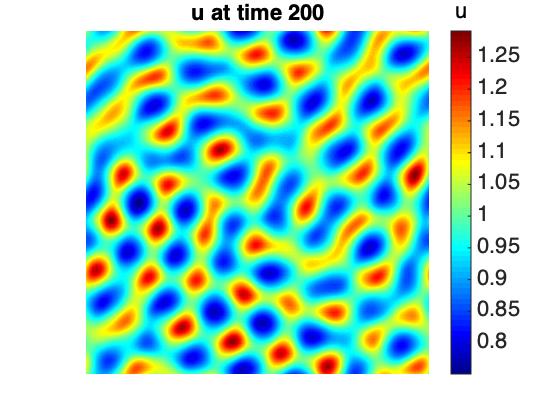}}
\subfigure[ $u$ at time 300 ]{\includegraphics[width=0.22\textwidth]{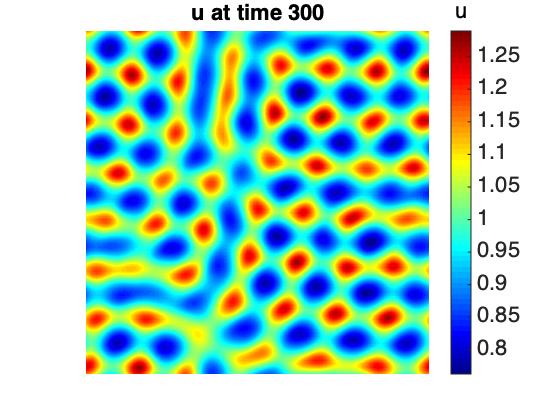}}
\subfigure[ $u$ at time 400 ]{\includegraphics[width=0.22\textwidth]{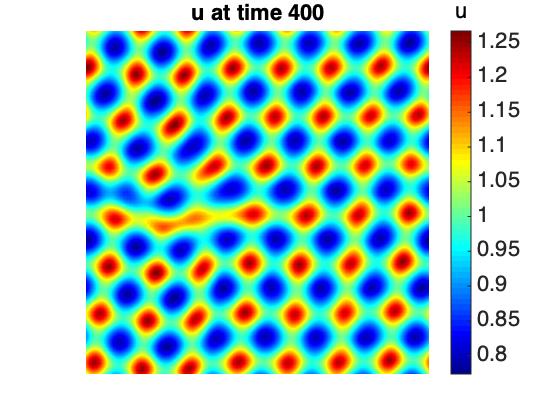}}\\

\subfigure[ $u$ at time 500 ]{\includegraphics[width=0.22\textwidth]{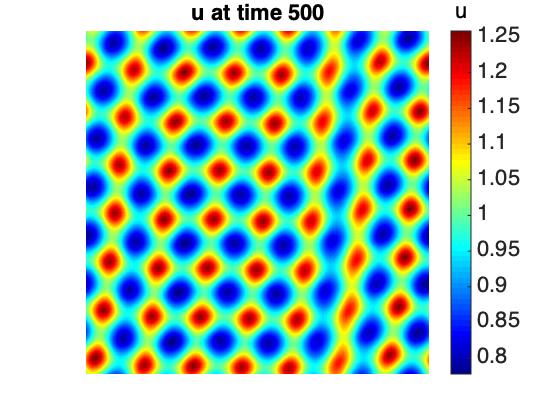}}
\subfigure[ $u$ at time 600 ]{\includegraphics[width=0.22\textwidth]{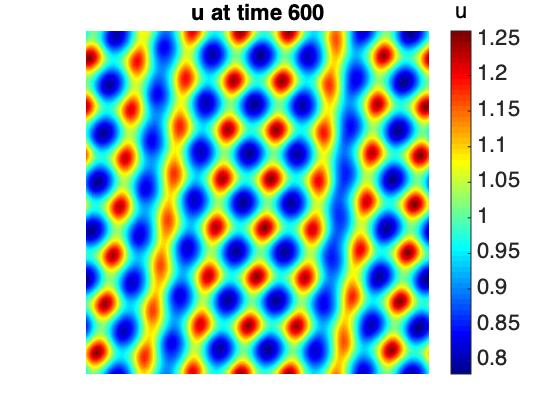}}
\subfigure[ $u$ at time 700 ]{\includegraphics[width=0.22\textwidth]{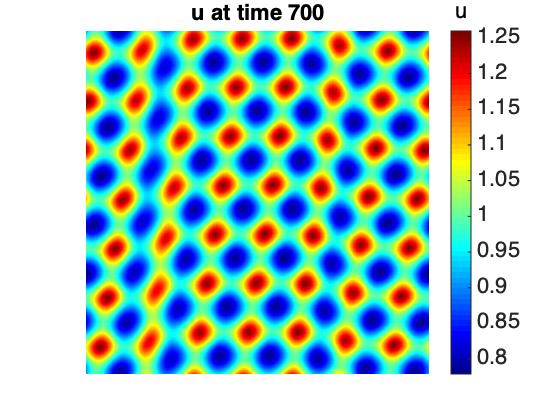}}
\subfigure[ $u$ at time 800 ]{\includegraphics[width=0.22\textwidth]{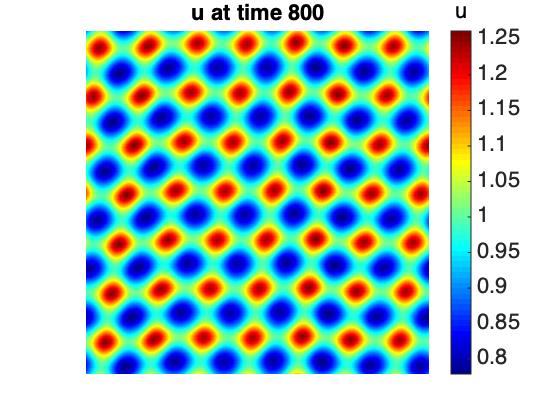}}\\

\caption{The dynamic process for  bacterial density  with large initial distribution and $s=-15$, $g=0.1$.}
\label{RE_p1_1u}
\end{figure}

\begin{figure}[H]
\centering
\subfigure[$c$ at time 50]{\includegraphics[width=0.22\textwidth]{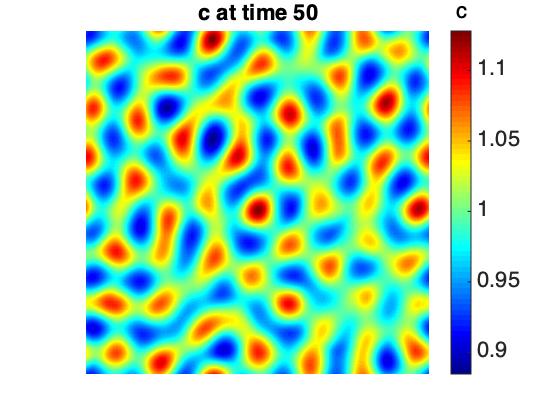}}
\subfigure[$c$ at time 200]{\includegraphics[width=0.22\textwidth]{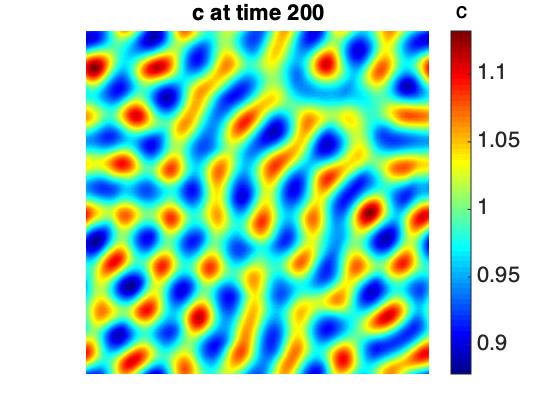}}
\subfigure[$c$ at time 300]{\includegraphics[width=0.22\textwidth]{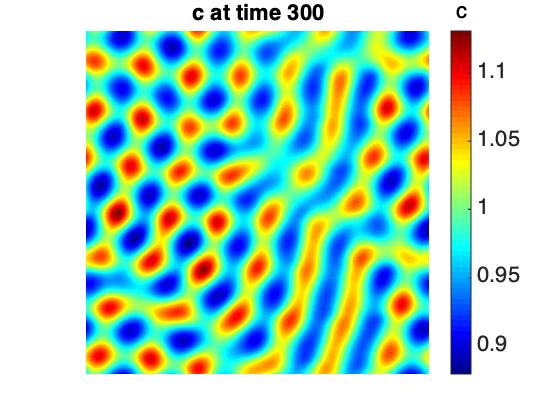}}
\subfigure[$c$ at time 400]{\includegraphics[width=0.22\textwidth]{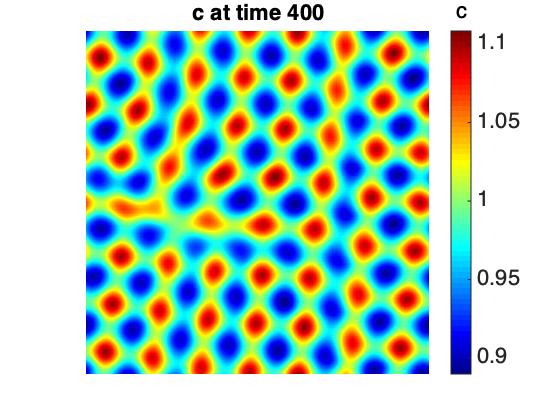}}\\

\subfigure[$c$ at time 500]{\includegraphics[width=0.22\textwidth]{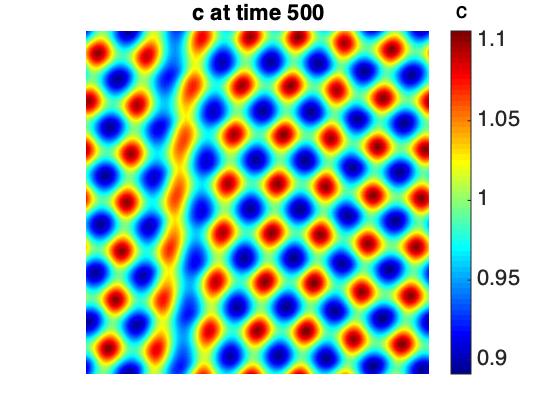}}
\subfigure[$c$ at time 600]{\includegraphics[width=0.22\textwidth]{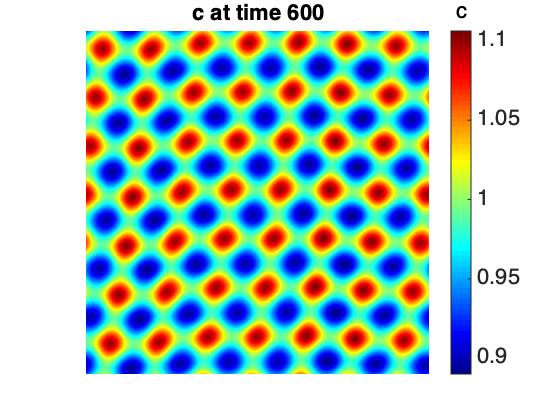}}
\subfigure[$c$ at time 700]{\includegraphics[width=0.22\textwidth]{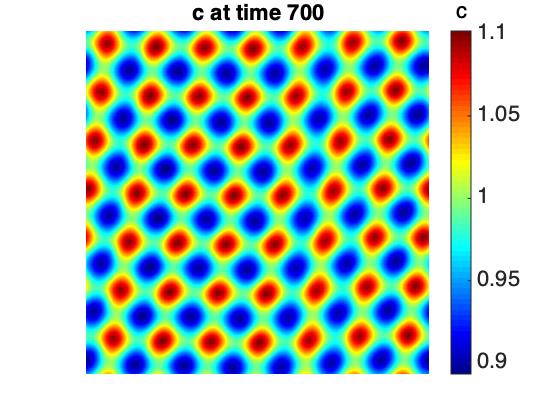}}
\subfigure[$c$ at time 800]{\includegraphics[width=0.22\textwidth]{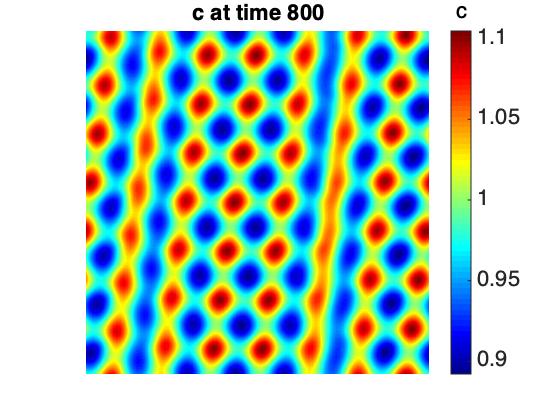}}\\
\caption{The dynamic process for density of self-secreted chemical  with large initial distribution and $s=-15$, $g=0.1$.}
\label{RE_p1_1c}
\end{figure}

\begin{figure}[H]
\centering
\subfigure[${\bf{p}}$ at time 50 ]{\includegraphics[width=0.22\textwidth]{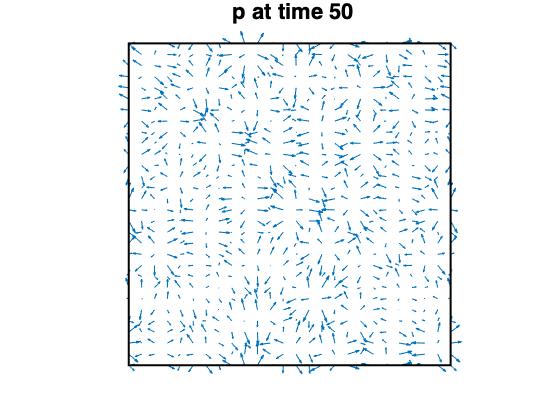}}
\subfigure[${\bf{p}}$ at time 200 ]{\includegraphics[width=0.22\textwidth]{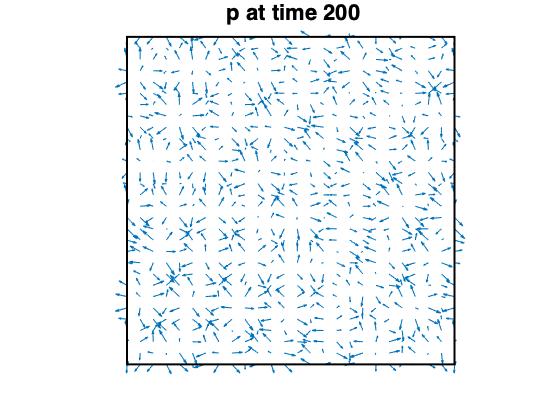}}
\subfigure[${\bf{p}}$ at time 300 ]{\includegraphics[width=0.22\textwidth]{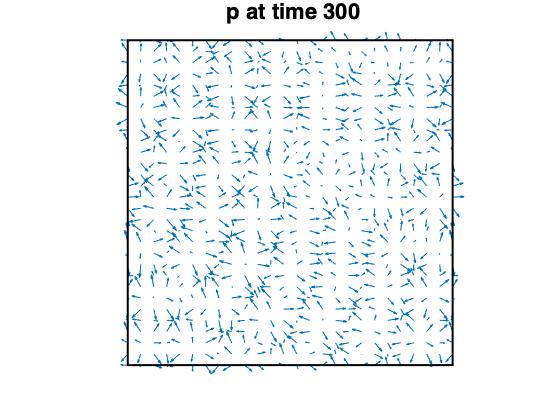}}
\subfigure[${\bf{p}}$ at time 400 ]{\includegraphics[width=0.22\textwidth]{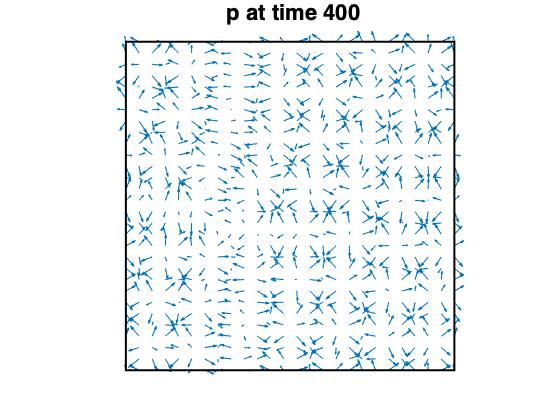}}\\

\subfigure[${\bf{p}}$ at time 500 ]{\includegraphics[width=0.22\textwidth]{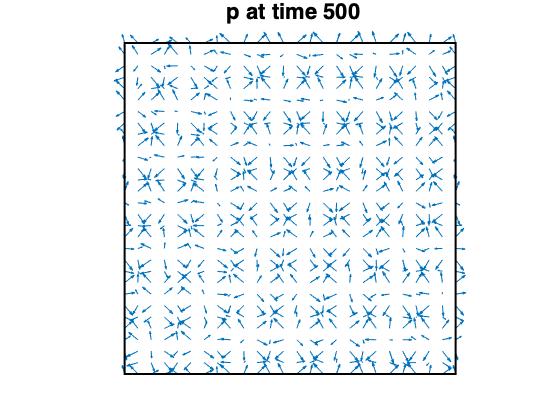}}
\subfigure[${\bf{p}}$ at time 600 ]{\includegraphics[width=0.22\textwidth]{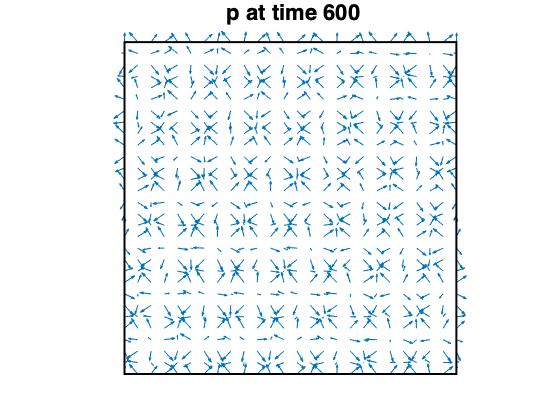}}
\subfigure[${\bf{p}}$ at time 700 ]{\includegraphics[width=0.22\textwidth]{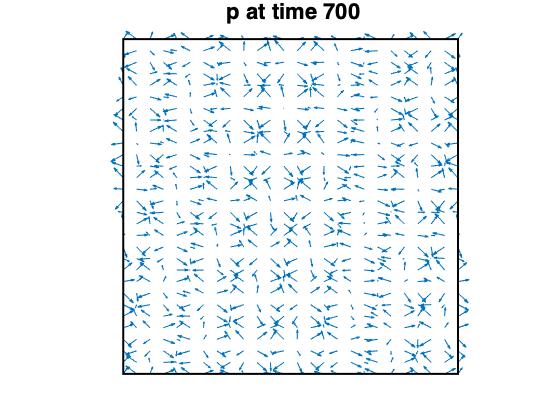}}
\subfigure[${\bf{p}}$ at time 800 ]{\includegraphics[width=0.22\textwidth]{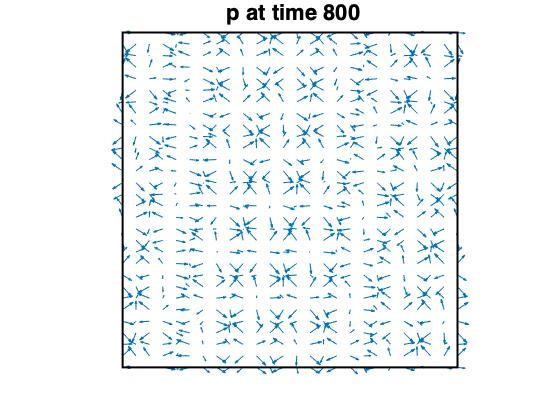}}\\

\caption{The dynamic process for polarization with large initial distribution and $s=-15$, $g=0.1$.}
\label{RE_p1_1p}
\end{figure}

\subsection{Chemorepulsion case 3}

In this subsection, the initial distribution and all the parameters are given as chemorepulsion case 1, except $s=-25$, $g=1$. The evolution of the process is given as  in Fig.\ref{RE2_p1u}-\ref{RE2_p1p}. We can find that clustering and wave pattern formation appear, and finally both of them form one  special order  and arrive at the  equilibrium state.
Comparing with case 1, the value of $g$ increases to 10 times. Which makes the diameter of clustering is smaller and the time of wave pattern formation appear early. The clustering form one line along the diagonal direction before the wave appear. And in the counter diagonal direction, part of the clustering gathering in the strip as represented in Fig.\ref{RE2_p1u}e-h.

\begin{figure}[H]
\centering
\subfigure[ $u$ at time 10 ]{\includegraphics[width=0.22\textwidth]{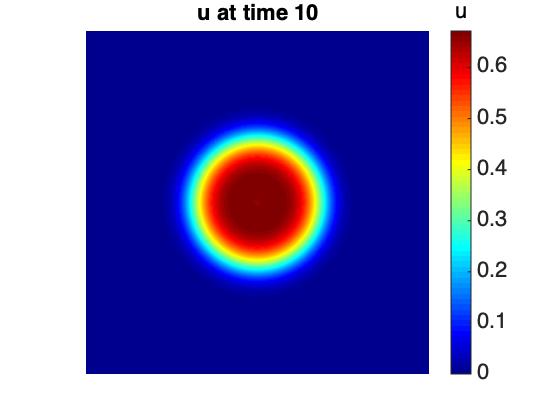}}
\subfigure[ $u$ at time 30 ]{\includegraphics[width=0.22\textwidth]{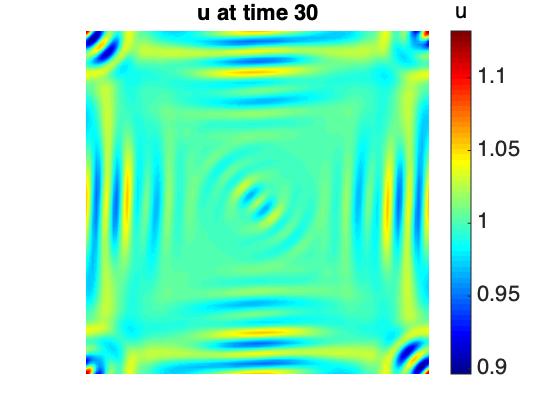}}
\subfigure[ $u$ at time 50 ]{\includegraphics[width=0.22\textwidth]{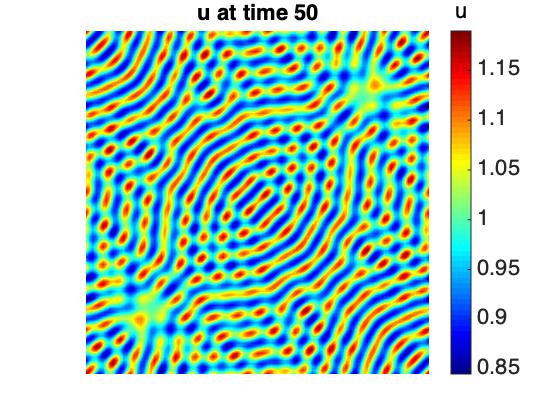}}
\subfigure[ $u$ at time 70 ]{\includegraphics[width=0.22\textwidth]{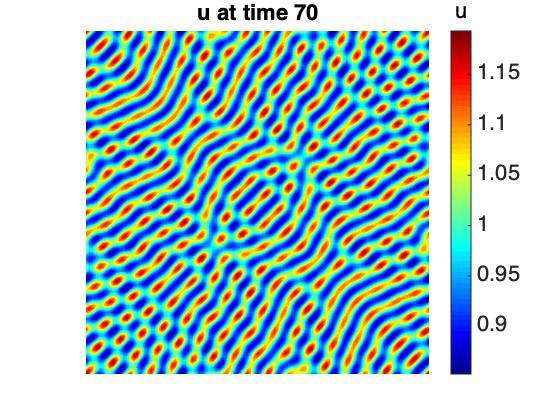}}\\

\subfigure[ $u$ at time 150 ]{\includegraphics[width=0.22\textwidth]{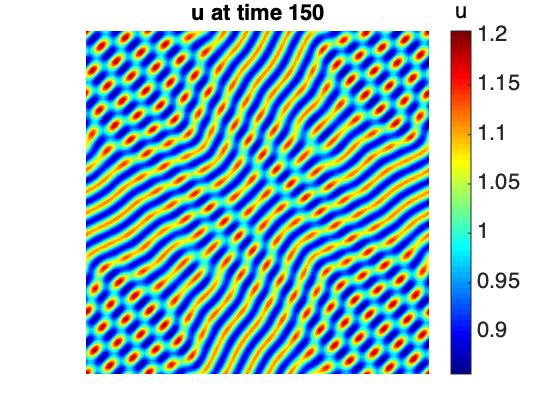}}
\subfigure[ $u$ at time 300 ]{\includegraphics[width=0.22\textwidth]{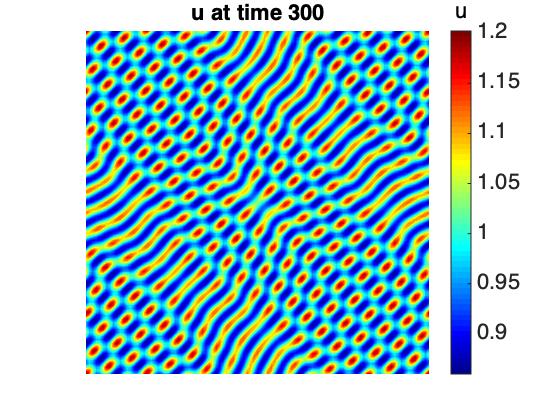}}
\subfigure[ $u$ at time 700 ]{\includegraphics[width=0.22\textwidth]{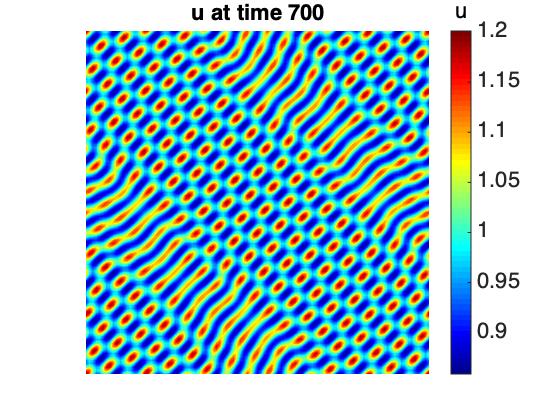}}
\subfigure[ $u$ at time 800 ]{\includegraphics[width=0.22\textwidth]{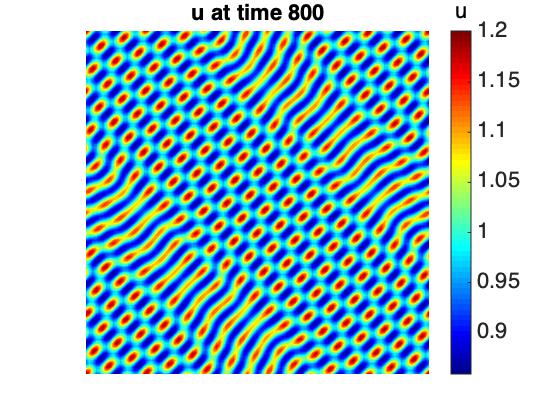}}\\

\caption{The dynamic process for bacterial density with small initial distribution and  $s=-25$, $g=1$.}
\label{RE2_p1u}
\end{figure}

\begin{figure}[H]
\centering
\subfigure[$c$ at time 10]{\includegraphics[width=0.22\textwidth]{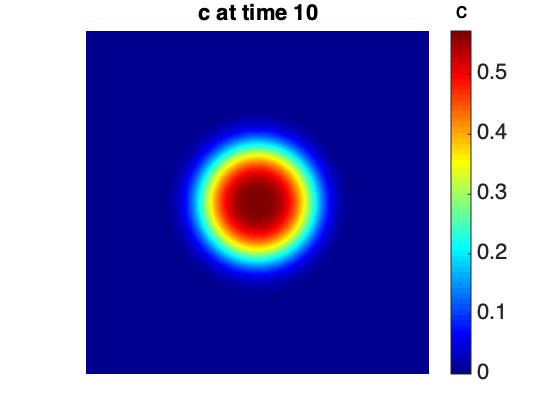}}
\subfigure[$c$ at time 30]{\includegraphics[width=0.22\textwidth]{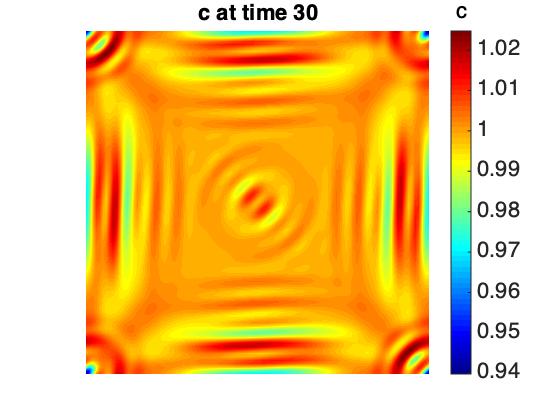}}
\subfigure[$c$ at time 50]{\includegraphics[width=0.22\textwidth]{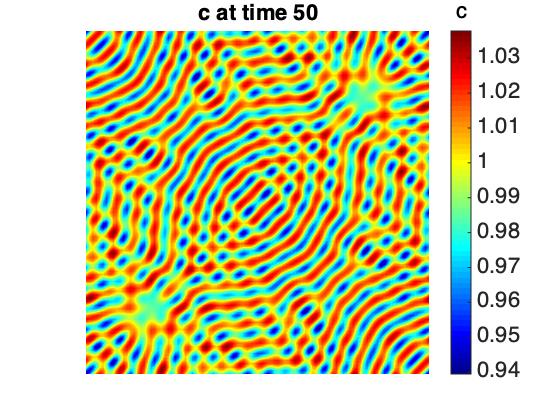}}
\subfigure[$c$ at time 70]{\includegraphics[width=0.22\textwidth]{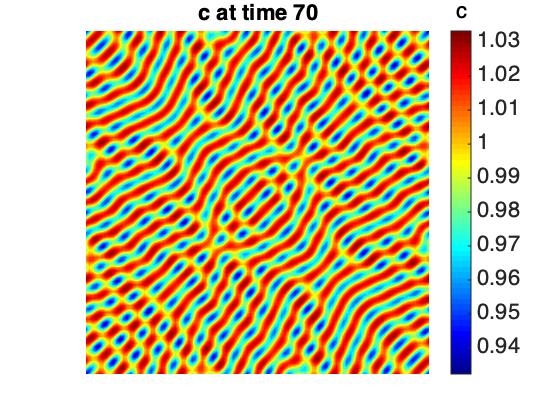}}\\

\subfigure[$c$ at time 150]{\includegraphics[width=0.22\textwidth]{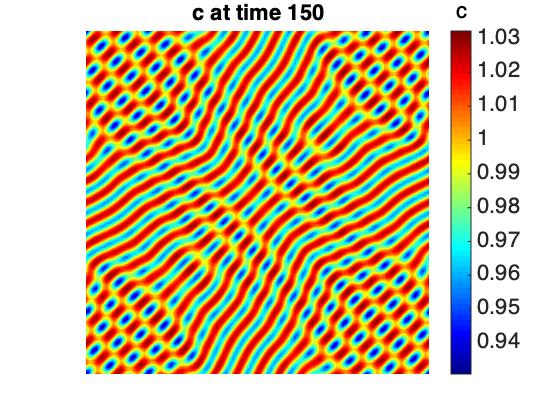}}
\subfigure[$c$ at time 300]{\includegraphics[width=0.22\textwidth]{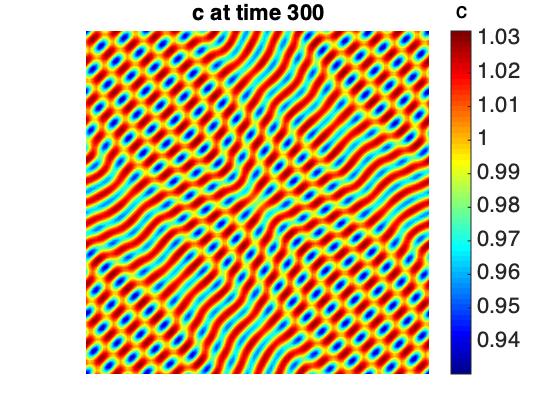}}
\subfigure[$c$ at time 700]{\includegraphics[width=0.22\textwidth]{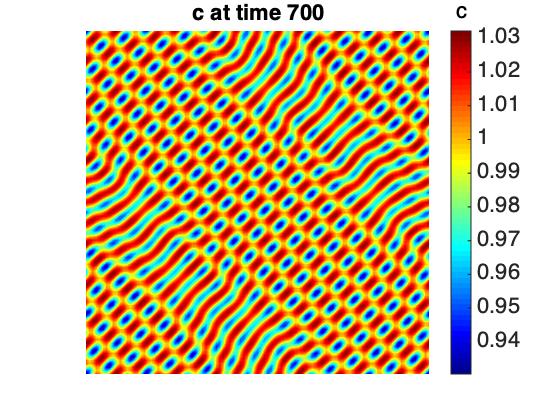}}
\subfigure[$c$ at time 800]{\includegraphics[width=0.22\textwidth]{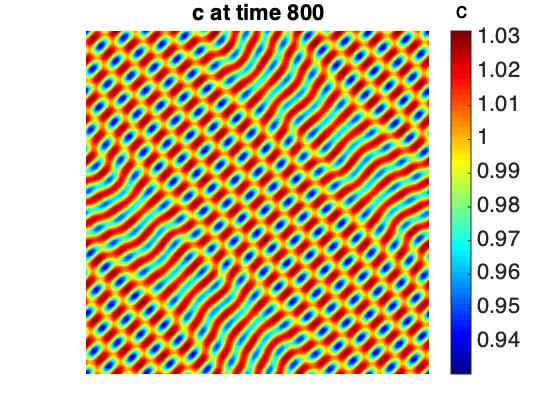}}\\
\caption{The dynamic process for density of self-secreted chemical  with small initial distribution and  $s=-25$, $g=1$.}
\label{RE2_p1c}
\end{figure}

\begin{figure}[H]
\centering
\subfigure[${\bf{p}}$ at time 10 ]{\includegraphics[width=0.22\textwidth]{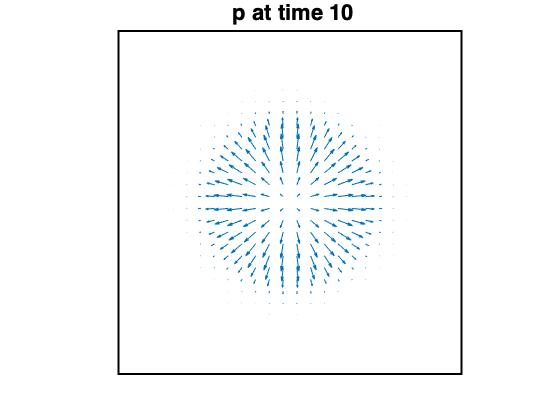}}
\subfigure[${\bf{p}}$ at time 30 ]{\includegraphics[width=0.22\textwidth]{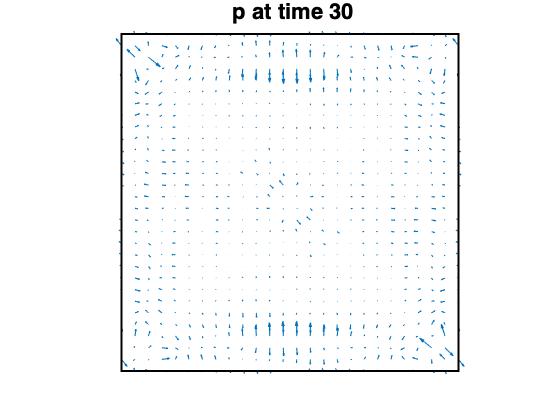}}
\subfigure[${\bf{p}}$ at time 50 ]{\includegraphics[width=0.22\textwidth]{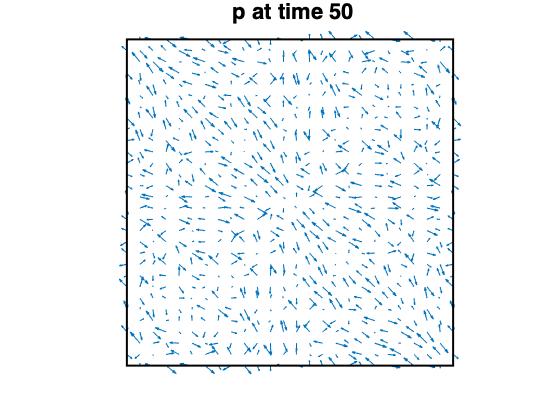}}
\subfigure[${\bf{p}}$ at time 70 ]{\includegraphics[width=0.22\textwidth]{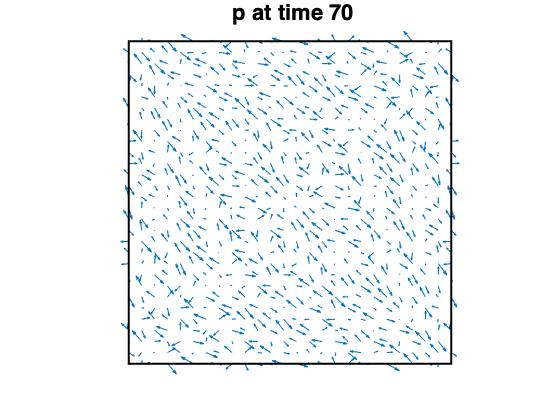}}\\

\subfigure[${\bf{p}}$ at time 150 ]{\includegraphics[width=0.22\textwidth]{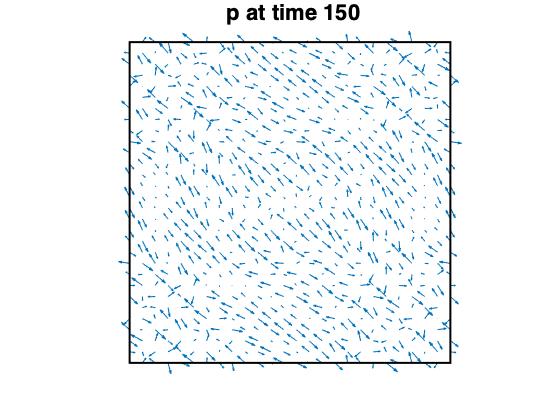}}
\subfigure[${\bf{p}}$ at time 300 ]{\includegraphics[width=0.22\textwidth]{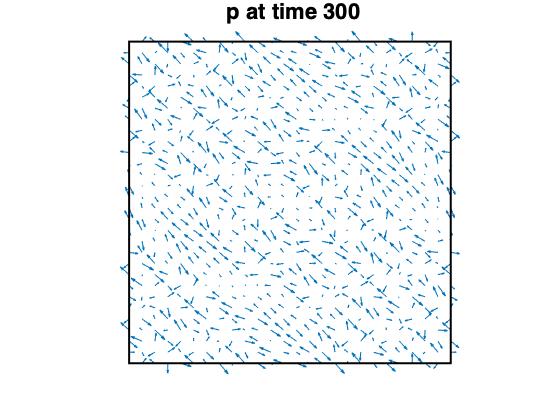}}
\subfigure[${\bf{p}}$ at time 700 ]{\includegraphics[width=0.22\textwidth]{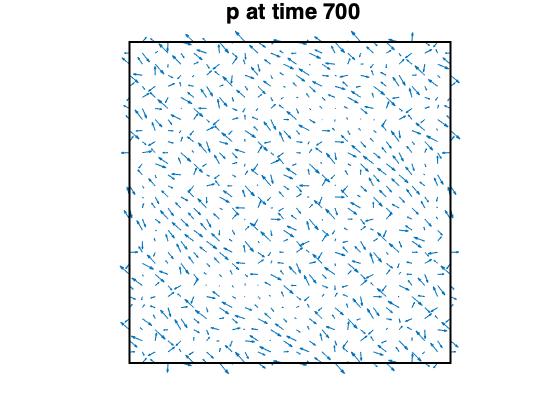}}
\subfigure[${\bf{p}}$ at time 800 ]{\includegraphics[width=0.22\textwidth]{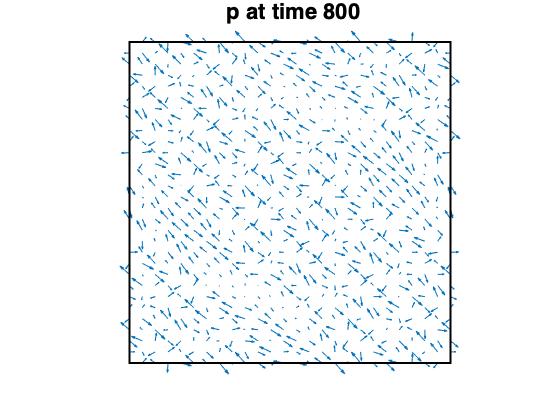}}\\

\caption{The dynamic process for polarization with small initial distribution and  $s=-25$, $g=1$.}
\label{RE2_p1p}
\end{figure}

\subsection{Chemorepulsion case 4}

For simulating the  chemorepulsion case, here we choose the initial condition with a small perturbation of the uniform state $(1,1,{\bf{0}})$ which is the same as chemorepulsion case 2. 
The value of parameters is the same as chemorepulsion case 3. The evolution of the process is obtained as in Fig.\ref{RE2_p1_1u}-\ref{RE2_p1_1p}. We can also find that clustering and wave pattern formation appear, and finally reach the equilibrium state. Comparing with case 3, at the beginning  the bacteria gather in the clustering, then quickly form the strip, and the wave pattern formation arise. In the equilibrium state, the strip pattern formation along the counter diagonal direction could break up. Which means the bacteria far away the breaking point. This is the repulsion due to the self-secreted chemical. Which confirms the theoretical analysis  in the physical background as presented in \cite{MP2018, Liebchen2015}.

\begin{figure}[H]
\centering
\subfigure[ $u$ at time 10 ]{\includegraphics[width=0.22\textwidth]{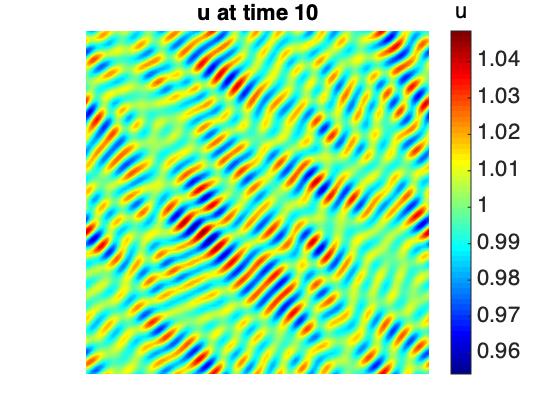}}
\subfigure[ $u$ at time 30 ]{\includegraphics[width=0.22\textwidth]{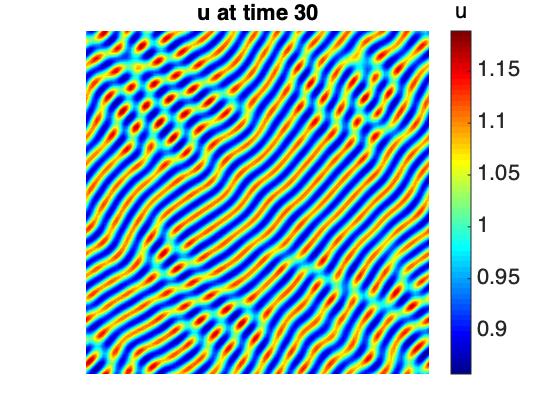}}
\subfigure[ $u$ at time 50 ]{\includegraphics[width=0.22\textwidth]{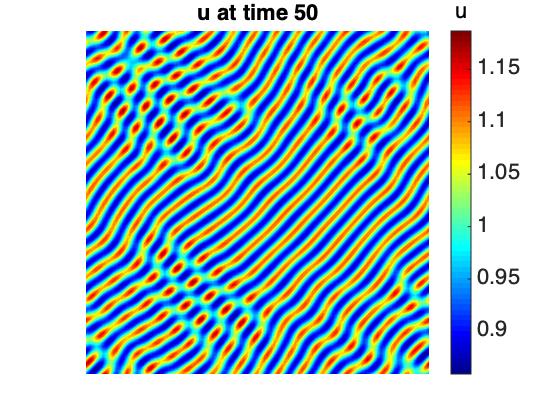}}
\subfigure[ $u$ at time 70 ]{\includegraphics[width=0.22\textwidth]{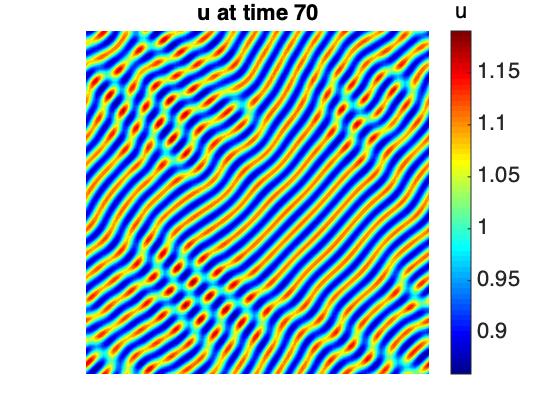}}\\

\subfigure[ $u$ at time 150 ]{\includegraphics[width=0.22\textwidth]{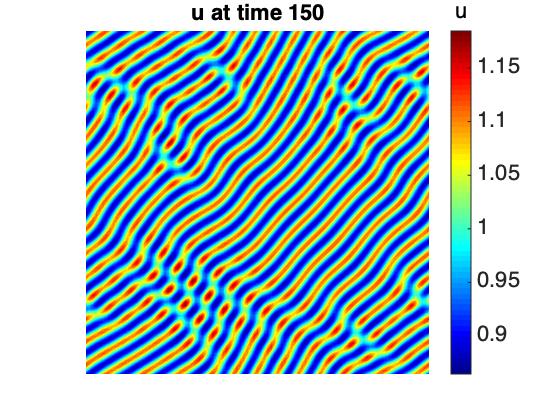}}
\subfigure[ $u$ at time 300 ]{\includegraphics[width=0.22\textwidth]{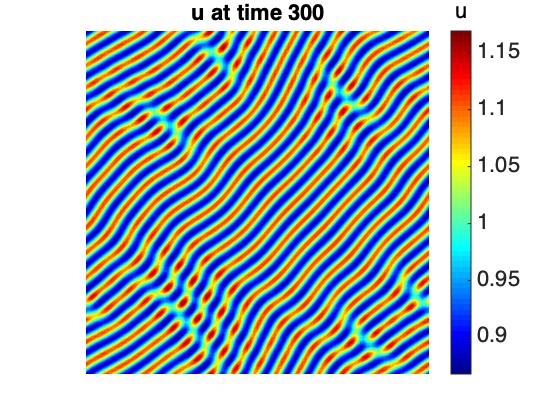}}
\subfigure[ $u$ at time 700 ]{\includegraphics[width=0.22\textwidth]{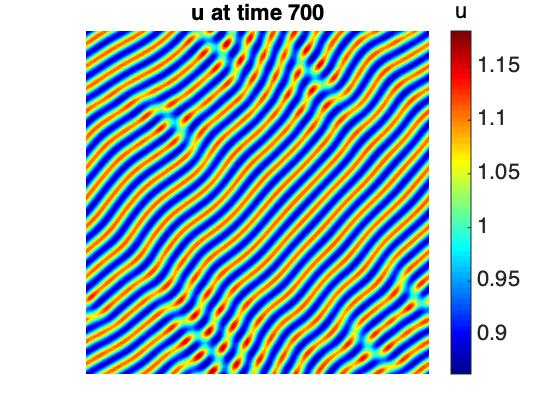}}
\subfigure[ $u$ at time 800 ]{\includegraphics[width=0.22\textwidth]{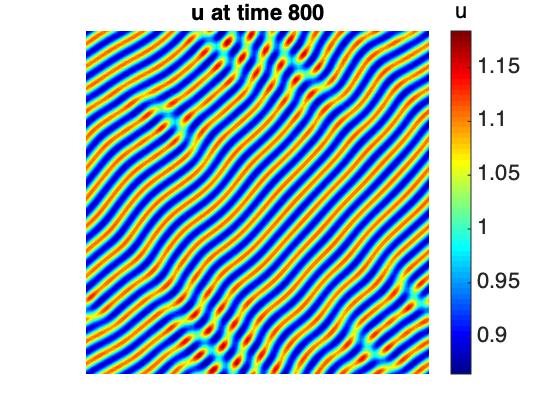}}\\

\caption{The dynamic process for
bacterial density with large initial distribution and  $s=-25$, $g=1$.}
\label{RE2_p1_1u}
\end{figure}

\begin{figure}[H]
\centering
\subfigure[$c$ at time 10]{\includegraphics[width=0.22\textwidth]{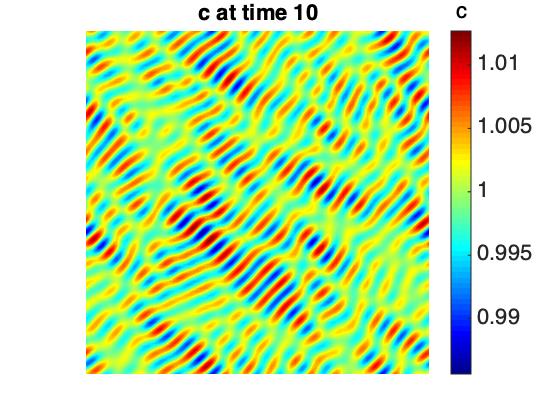}}
\subfigure[$c$ at time 30]{\includegraphics[width=0.22\textwidth]{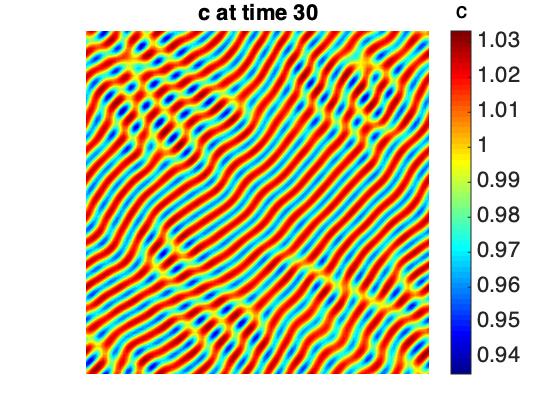}}
\subfigure[$c$ at time 50]{\includegraphics[width=0.22\textwidth]{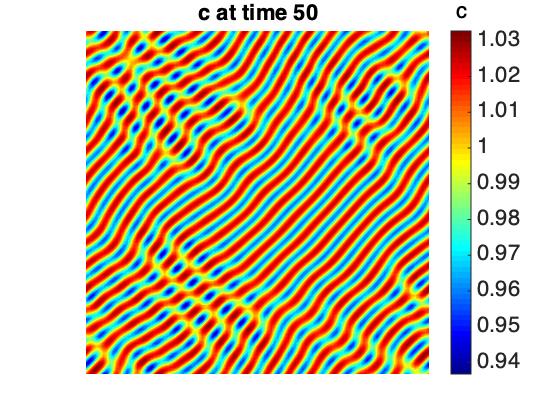}}
\subfigure[$c$ at time 70]{\includegraphics[width=0.22\textwidth]{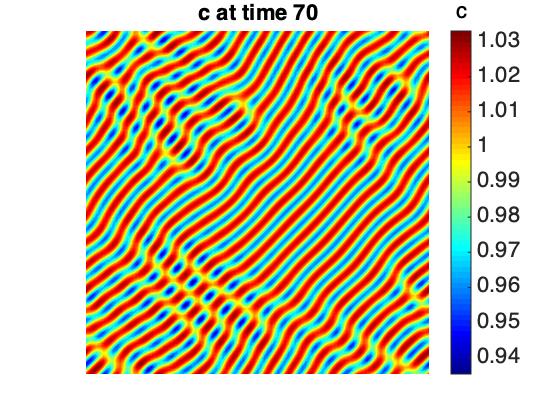}}\\

\subfigure[$c$ at time 150]{\includegraphics[width=0.22\textwidth]{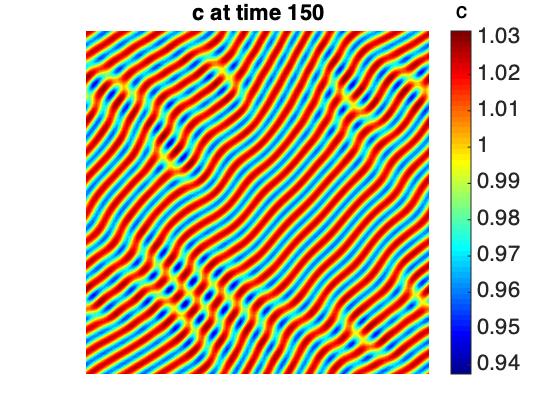}}
\subfigure[$c$ at time 300]{\includegraphics[width=0.22\textwidth]{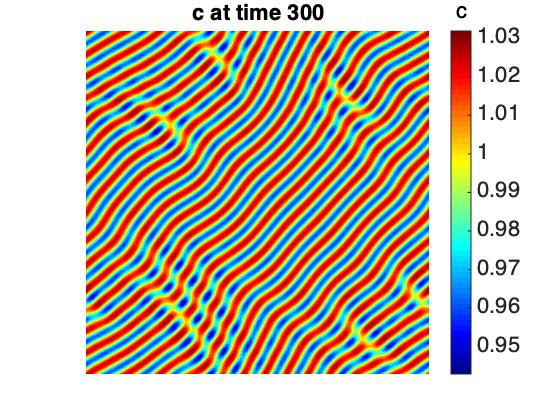}}
\subfigure[$c$ at time 700]{\includegraphics[width=0.22\textwidth]{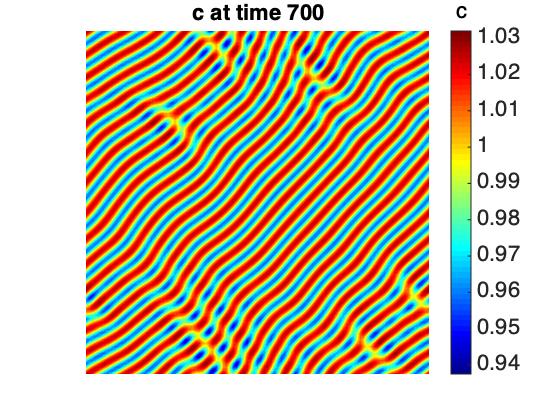}}
\subfigure[$c$ at time 800]{\includegraphics[width=0.22\textwidth]{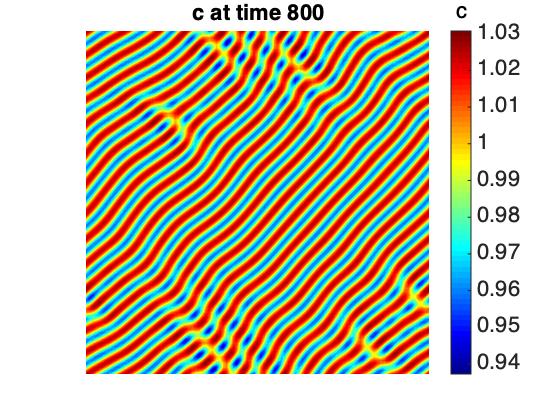}}\\
\caption{The dynamic process for density of self-secreted chemical  with large initial distribution and  $s=-25$, $g=1$.}
\label{RE2_p1_1c}
\end{figure}

\begin{figure}[H]
\centering
\subfigure[${\bf{p}}$ at time 10 ]{\includegraphics[width=0.22\textwidth]{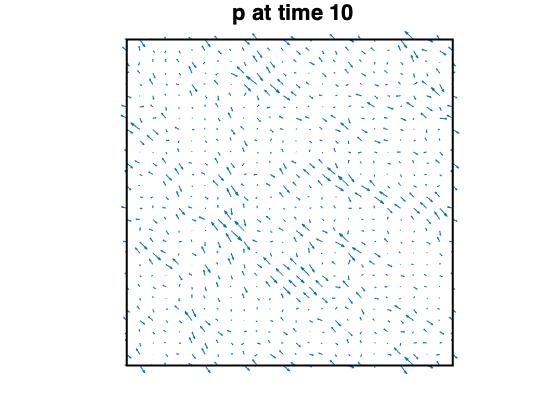}}
\subfigure[${\bf{p}}$ at time 30 ]{\includegraphics[width=0.22\textwidth]{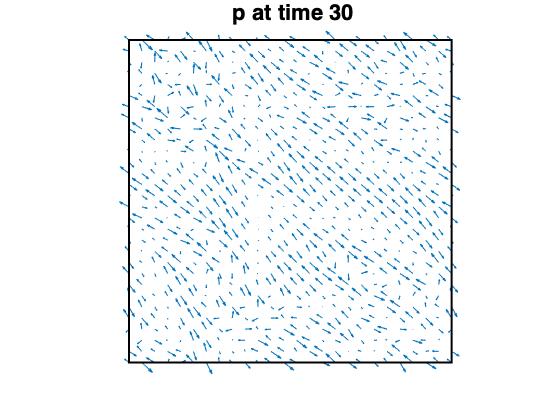}}
\subfigure[${\bf{p}}$ at time 50 ]{\includegraphics[width=0.22\textwidth]{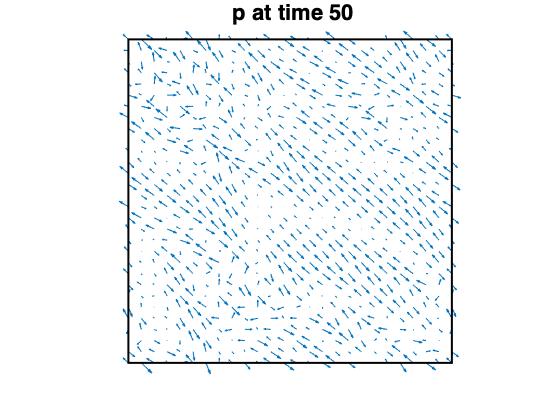}}
\subfigure[${\bf{p}}$ at time 70 ]{\includegraphics[width=0.22\textwidth]{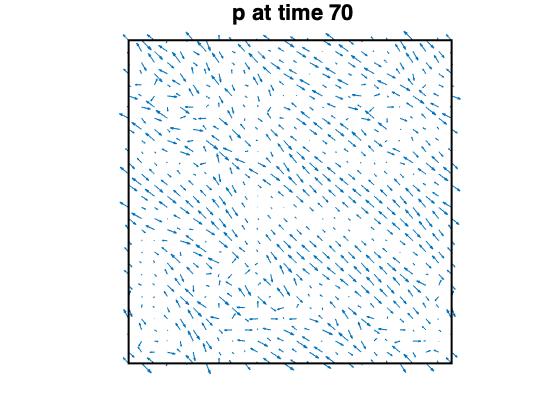}}\\

\subfigure[${\bf{p}}$ at time 150 ]{\includegraphics[width=0.22\textwidth]{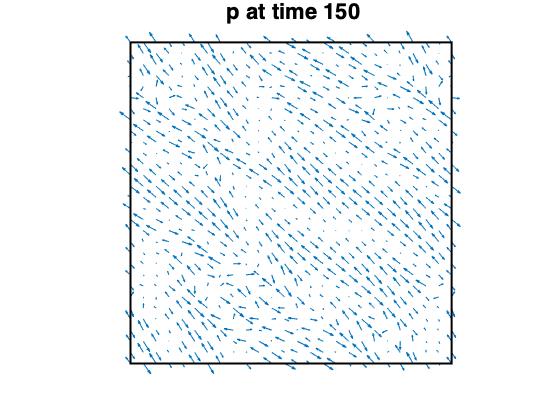}}
\subfigure[${\bf{p}}$ at time 300 ]{\includegraphics[width=0.22\textwidth]{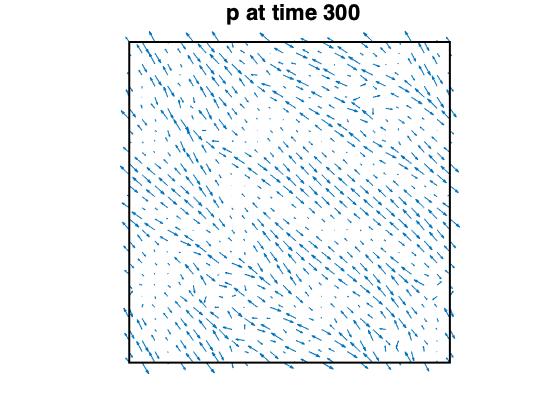}}
\subfigure[${\bf{p}}$ at time 700 ]{\includegraphics[width=0.22\textwidth]{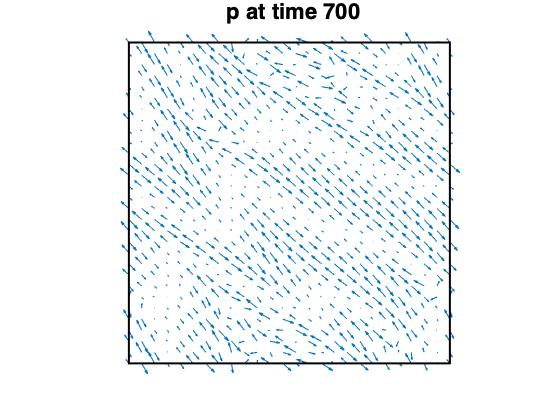}}
\subfigure[${\bf{p}}$ at time 800 ]{\includegraphics[width=0.22\textwidth]{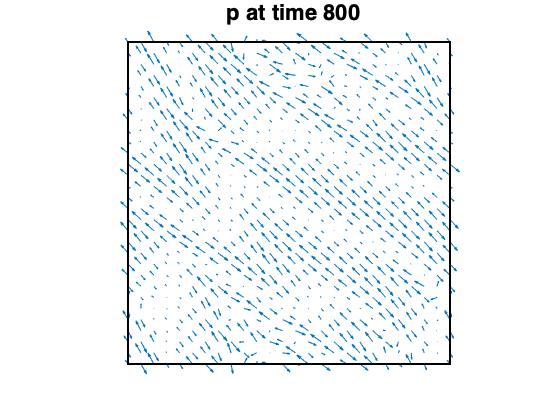}}\\

\caption{The dynamic process for polarization with large initial distribution and  $s=-25$, $g=1$.}
\label{RE2_p1_1p}
\end{figure}

\section{Conclusion}
For simulating growth-mediated autochemotactic pattern formation in self-propelling bacteria, we propose one   combined modified characteristic Galerkin finite element method.  In this method, the discrete system is divided into three separated symmetric positive definite sub-problems, and  keeps  mass balance wholly. The convergence of the proposed method is studied and the error estimate is also derived. We present four chemorepulsion cases. Under the small value of grow rate, before the wave pattern formation arise, the clustering and strip appear. In the large value of grow rate, the  bacteria gather and quickly move to the wave pattern formation and arrive at the equilibrium state. In this process, the effect of repulsion from self-secreted  chemical  and the initial distribution also affect the structure of the pattern before wave arising. Which confirms the theoretical analysis  in the different regimes and also brings up new results for expanding the physical mechanism.

\section*{Acknowledgments}

 J. Zhang's work was supported by the Fundamental Research Funds for the Central Universities (20CX05011A). M. Jiang’s work was supported partially by the Natural Science Foundation of Shandong 
Province(Grant number ZR2021QA018).  J. Zhu's work was partially supported by the National Council for Scientific and Technological Development of Brazil (CNPq).  X. Yu's work was supported partially by the National Natural Science Foundation of China (Grant No. 12071046) . L. Bevilacqua's work was supported partially by CNPq/TWAS Grant, the COPPE/CAPES Grant 001, and the USP/IEA visiting research program. 

%\section*{References}

\end{document}